\documentclass[a4paper, 10pt]{article}
\usepackage[italian, english]{babel}
\usepackage{graphicx}
\usepackage{epstopdf}
\usepackage[T1]{fontenc} 
\usepackage[utf8]{inputenc}
\usepackage{hyperref}
\usepackage{amsthm}
\usepackage{amsfonts}
\usepackage{xcolor}

\DeclareFontFamily{U}{mathx}{\hyphenchar\font45}
\DeclareFontShape{U}{mathx}{m}{n}{
      <5> <6> <7> <8> <9> <10>
      <10.95> <12> <14.4> <17.28> <20.74> <24.88>
      mathx10
      }{}
\DeclareSymbolFont{mathx}{U}{mathx}{m}{n}
\DeclareFontSubstitution{U}{mathx}{m}{n}
\DeclareMathAccent{\widecheck}{0}{mathx}{"71}
\DeclareMathAccent{\wideparen}{0}{mathx}{"75}

\usepackage{spverbatim}
\usepackage{appendix}

\usepackage{amssymb}
\usepackage{mathtools}

\usepackage{geometry}
\graphicspath{ {media/} }

\newtheorem{lemma}{Lemma}

\newtheorem{teo}{Theorem}
\newtheorem{prop}{Proposition}

\newtheorem{con}{Conjecture}

\newtheorem{rem}{Remark}

\makeatletter
\g@addto@macro\bfseries{\boldmath}
\makeatother
\date{}
\usepackage{tabularx,booktabs}

\title{Improved Bounds for $(b, k)$-hashing}
\author{Stefano Della Fiore\footnote{DII, University of Brescia, Via Branze 38,  25123 Brescia, Italy.  Emails: s.dellafiore001@unibs.it, marco.dalai@unibs.it. 
}, Simone Costa\footnote{DICATAM, University of Brescia, Via Branze 43,  25123  Brescia, Italy.  Email: simone.costa@unibs.it. \newline \indent This paper was presented in part at ISIT 2021.}, Marco Dalai\footnotemark[1]}
\begin{document}
\maketitle
\begin{abstract}
For fixed integers $n$ and $b\geq k$, let $A(b,k,n)$ the largest size of a subset of $\{1,2,\ldots,b\}^n$ such that, for any $k$ distinct elements in the set, there is a coordinate where they all differ. Bounding $A(b,k,n)$ is a problem of relevant interest in information theory and in computer science, relating to the zero-error capacity with list decoding and with the study of $(b, k)$-hash families of functions.  It is known that, for fixed $b$ and $k$, $A(b,k,n)$ grows exponentially in $n$. In this paper, we determine new exponential upper bounds for different values of $b$ and $k$.

A first bound on $A(b,k,n)$ for general $b$ and $k$ was derived by Fredman and Koml\'os in the '80s and improved for certain $b\neq k$ by K\"orner and Marton and by Arikan. Only very recently better bounds were derived for general $b$ and $k$ by Guruswami and Riazanov, while
stronger results for small values of $b=k$ were obtained by Arikan, by Dalai, Guruswami and Radhakrishnan, and by Costa and Dalai. In this paper, we strengthen the bounds for some specific values of $b$ and $k$. Our contribution is a new computational method for obtaining upper bounds on the values of a quadratic form defined over discrete probability distributions in arbitrary dimensions, which emerged as a central ingredient in recent works. The proposed method reduces an infinite-dimensional problem to a finite one, which we manage to further simplify by means of  a series of optimality conditions.
\end{abstract}
\section{Introduction}

The problem considered in this paper has a twofold history that connects it naturally with combinatorial aspects of computer science and information theory. Let $b$, $k$, and $n$ be integers and let $C$ be a subset of $\{1,2,\ldots,b\}^n$ with the property that for any $k$ distinct elements we can find a coordinate where they all differ. Such a set can be interpreted, by looking at it coordinate-wise, as a family of $n$ hashing functions on some universe of size $|C|$.  The required property then says that the family is a perfect hash family, that is, any $k$ elements in the universe are $k$-partitioned by at least one function. Alternatively, $C$ can be interpreted as a code of rate $\log(|C|)/n$ for communication over a channel with $b$ inputs. Assume that the channel is a $b/(k-1)$ channel, meaning that any $k-1$ of the $b$ inputs share one output but no $k$ distinct inputs do (see Figure \ref{fig:channel}). The required property for $C$ is what is needed for the receiver to always be able to produce a list of $k-1$ codewords of $C$ which must necessarily include the one that was sent; that is, zero-error communication with $(k-1)$-list decoding is possible. Indeed, the condition implies that any $k$ codewords use, in at least one coordinate,  $k$ different symbols, and one of them will not be compatible with the received symbol in that coordinate. We refer the reader to \cite{Elias1}, \cite{FredmanKomlos}, \cite{Korner3}, \cite{nilli} for an overview of the more general context of this problem. Some recent important results in a different asymptotic setting can be found in \cite{Jaikumar2}.

\begin{figure}[t]
\centering
\includegraphics[scale=0.8]{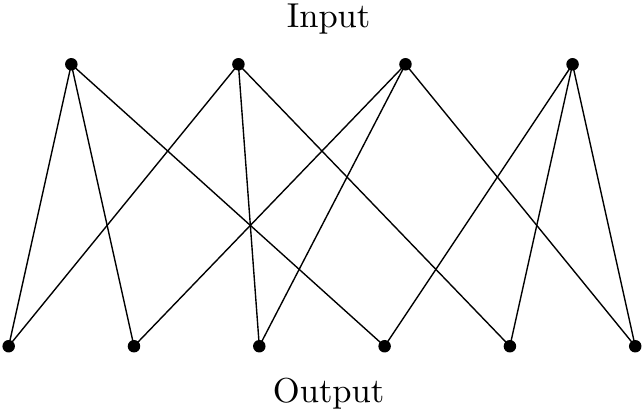}
\caption{A $4/2$ channel. Edges represent positive probabilities. Here, zero-error communication is possible  when decoding with list-size equal to $2$.}
\label{fig:channel}
\end{figure}

We will call any subset $C$ of $\{1,2,\ldots,b\}^n$ with the described property a $(b, k)$-hash code. For the reasons mentioned above, bounding the size of $(b, k)$-hash codes is a combinatorial problem that has been of interest in both computer science and information theory. Let $A(b,k,n)$ be the largest size of such a code. It is known that for fixed $b$ and $k$, $A(b,k,n)$ grows exponentially in $n$, and a challenging problem consists in bounding the exponent. We will thus study the quantity
\begin{equation}
R_{(b,k)}=\limsup_{n\to \infty}\frac{1}{n}\log A(b,k,n)\,.
\end{equation}
Note that, throughout the paper,  all logarithms are to base 2. 

Few lower bounds on $R_{(b,k)}$ are known. First results in this sense were given by \cite{FredmanKomlos}, \cite{Elias1} and a better bound was derived by \cite{Korner2} for $b=k=3$. More recently, new lower bounds were derived in \cite{xing-yuang} for infinitely many other values of $k$. 
The first, landmark result concerning the upper bounds was obtained by Fredman-Koml\'os \cite{FredmanKomlos}, who showed that 
\begin{equation}
R_{(b,k)} \leq  \frac{b^{\underline{k-1}}}{b^{k-1}} \log (b-k+2)\,,
\label{eq:fredmankomlos}
\end{equation}
where $b^{\underline{k-1}} = b (b-1) \cdots (b-k+2)$.
Progress has since been rare.
A generalization of the bound given in equation \eqref{eq:fredmankomlos} was derived by K\"orner and Marton \cite{Korner2} in the form
\begin{equation}
R_{(b,k)} \leq \min_{2 \leq j \leq k-2} \frac{b^{\underline{j+1}}}{b^{j+1}} \log \frac{b-j}{k-j-1}\,.
\label{eq:kornermarton}
\end{equation}
Nilli \cite{nilli} provided an elementary proof of \eqref{eq:kornermarton} without considerations of graph entropy or hypergraph entropy.
This bound was further improved for different values of $b$ and $k$ by Arikan \cite{Arikan}.
In the case $b=k$, an improvement was first obtained for $k=4$ in \cite{Arikan2} and then in \cite{DalaiVenkatJaikumar}, \cite{DalaiVenkatJaikumar2}.  The latter only focuses on $b=k=4$, but the procedure can be extended to general $b$ and $k$. As shown in the next sections, it leads to the following bound.
\begin{lemma}
For general $b$ and $k$, we have
\begin{equation}
	R_{(b,k)}  \leq \left(\frac{1}{\log b} + \frac{b^2}{(b^2-3b+2) \log \frac{b-2}{k-3}}  \right)^{-1}.
	\label{eq:dalaivenkatbound}
\end{equation}
\label{dalaivenkat_lemma}
\end{lemma}
In \cite{venkat}, the authors prove that the Fredman-Koml\'os bound is not tight for any $b \geq k > 3$; explicit better values were given there for $b=k=5,6$, and for larger $b=k$ modulo a conjecture which is proved in \cite{costaDalai}, where further improvements are also obtained for $b=k=5, 6$. The case of $b\neq   k$  is not described in detail in \cite{venkat} but, as the authors mention, it is straightforward. We do not write here the bound since it has a complicated expression. 

In this paper, we attack some of the cases which appear not to be optimally handled by those methods. In particular, we build on the results obtained in \cite{costaDalai} and add an improvement that leads to better bounds for many pairs of $(b,k)$ values. The results of \cite{costaDalai} for $b=k$ were derived following an approach common to many recent works by introducing a symmetrization which reduces to the problem of bounding a quadratic form on probability distributions. We give a more general exposition for the general $b,k$ case, anticipating here the key lemma whose proof we give for completeness in the next section. Fix an integer $j$ in the range $2,\ldots,k-2$ and define, for probability vectors $p,q\in \mathbb{R}^b$, the function
\begin{equation}
\Psi_j(p;q) = \frac{1}{(b-j-1)!} \sum_{\sigma} p_{\sigma(1)}p_{\sigma(2)}\dots p_{\sigma(j)}q_{\sigma(j+1)} + q_{\sigma(1)}q_{\sigma(2)}\dots q_{\sigma(j)}p_{\sigma(j+1)},
\label{eq:defPsi}
\end{equation}
where $\sigma$ ranges over all permutations of $\{1,2,\ldots,b\}$. Define then
\begin{equation}
\mathbf{M}_j=\sup_\lambda \sum_{p,q}\lambda_p\lambda_q \Psi_j(p;q)
\label{eq:QuadraticForm}
\end{equation}
where $\lambda$ ranges over all probability distributions on finite sets of probability vectors in $\mathbb{R}^b$, so that $\lambda_p$ is the probability associated to the probability vector $p$. Then, the following bound holds.
\begin{lemma}
For $j=2,\ldots,k-2$,
\begin{equation}
R_{(b,k)}\leq  \left(\frac{2}{\mathbf{M}_j \log \frac{b-j}{k-j-1} }+\frac{1}{\log\left(\frac{b}{j-1}\right)}\right)^{-1}\,.
\label{Rbk_Mj_bound}
\end{equation}
\label{Rbk_Mj_lemma}
\end{lemma}
The results in \cite{costaDalai} were obtained using in \eqref{Rbk_Mj_bound}, for $b=k$ and $j=k-2$, the upper bound 
\begin{equation}
\mathbf{M}_j \leq \max_{p,q}\Psi_j(p;q)\,.
\label{eq:PsiMax}
\end{equation}
A weakness in this bound comes from the fact that distributions $p$ and $q$ that maximize $\Psi_j(p;q)$ exhibit in many cases some opposing asymmetries, in the sense that they give higher probabilities to different symbols. When used as a replacement for \emph{each} of the pairs of $p$ and $q$ in \eqref{eq:QuadraticForm}, we have a rather conservative bound, because pairs $(p,q)$ which give high values for $\Psi_j(p;q)$ will give low values for $\Psi_j(p;p)$ and $\Psi_j(q;q)$, and equation \eqref{eq:QuadraticForm} contains a weighted contribution from all pairings of $p$ and $q$. In this paper, we present a computational method for obtaining more refined bounds on $\mathbf{M}_j$ for general $b, k$ values which lead to improvements on the best-known bounds on $R_{(b,k)}$ for many $b,k$ pairs.


In Table \ref{tab:bkbounds} we give a comparison between bounds \eqref{eq:dalaivenkatbound} and \eqref{eq:kornermarton}, the bounds of \cite{Arikan} and \cite{venkat} and our new bounds for different values of $b$ and $k$. In Table \ref{tab:dalaivenkatjaikumar} we show that for some  $(b,k)$-cases the bound \eqref{eq:dalaivenkatbound} is the best bound among all the current known bounds, in particular when $b$ is much larger than $k$. Finally, in Table \ref{tab:venkatriazanov} we provide some $(b,k)$-cases where the bound of \cite{venkat} is the current best known bound, in particular when $b$ and $k$ are large and nearly equal. Clearly, the cases reported in Tables \ref{tab:dalaivenkatjaikumar} and \ref{tab:venkatriazanov} are not exhaustive, but they have been properly selected to point out that our method does not always provide the best bounds.
The integers in the parentheses for bounds \cite{venkat}, \cite{Arikan} and \cite{Korner2} in Table \ref{tab:dalaivenkatjaikumar} represent the optimal value of a parameter which has the same role as $j$ in \eqref{eq:kornermarton}. When its value is not reported, as well as in Tables \ref{tab:bkbounds} and \ref{tab:venkatriazanov}, it is equal to $k-2$ for our bounds and for bounds of \cite{venkat}, \cite{Arikan} and \cite{Korner2}. Instead, for bound \eqref{eq:dalaivenkatbound} it is always equal to $2$.
 


\begin{table}[!htbp]
\footnotesize
\def\arraystretch{1.17}
\caption{Upper bounds on $R_{(b,k)}$. All numbers are rounded upwards.}
\centering
\begin{tabularx}{\linewidth}{c@{\extracolsep{\fill}}c@{\extracolsep{\fill}}c@{\extracolsep{\fill}}c@{\extracolsep{\fill}}c|@{\extracolsep{\fill}}c@{\extracolsep{\fill}}c@{\extracolsep{\fill}}c@{\extracolsep{\fill}}c@{\extracolsep{\fill}}c}
$(b,k)$ & Our method & \cite{Arikan} & \cite{venkat} & \cite{Korner2} & $(b,k)$ & Our method & \cite{Arikan} & \cite{venkat} & \cite{Korner2} \\
\hline
$(5,5)$ & \textbf{0.16894}$^1$  & 0.23560 & 0.19079 & 0.19200 & $(6,5)$ & \textbf{0.34512}$^1$ & 0.44149 & 0.43207 &  0.44027\\
$(6,6)$ & \textbf{0.08475}$^1$ & 0.15484 & 0.09228 & 0.09260 & $\underline{(7,6)}$ & \textbf{0.19897}$^2$ & 0.30554 & 0.23524 & 0.23765 \\
$\underline{(8,6)}$ & \textbf{0.31799}$^2$  & 0.44888 & 0.40330 & 0.41016 & $\underline{(9,6)}$ & \textbf{0.43237}$^2$  & 0.58303 & 0.58486 & 0.59455 \\
$\underline{(10,6)}$ & \textbf{0.53909}$^2$ & 0.73304 & 0.76977 & 0.78170 & $\underline{(11,6)}$ & \textbf{0.63766}$^2$ & 0.87038 & 0.95285 & 0.96640\\
$\underline{(12,6)}$ & \textbf{0.72848}$^2$ & 0.99588  & 1.13118 & 1.14584 & $\underline{(13,6)}$ & \textbf{0.81227}$^2$ & 1.11084  & 1.30322 & 1.31855\\
$\underline{(14,6)}$ & \textbf{0.88978}$^2$ & 1.21657  & 1.46822 & 1.48388 & $(7,7)$ & \textbf{0.04090}$^1$ & 0.09747 & 0.04279 & 0.04284\\
$\underline{(8,7)}$ & \textbf{0.10865}$^2$  & 0.20340 & 0.12134 & 0.12189 & $\underline{(9,7)}$ & \textbf{0.19054}$^2$ & 0.31204 & 0.22547 & 0.22761  \\
$\underline{(10,7)}$ & \textbf{0.27741}$^2$ & 0.41982 & 0.34615 & 0.35108  & $\underline{(11,7)}$ & \textbf{0.36424}$^2$  & 0.52472  & 0.47856 & 0.48538 \\
$\underline{(12,7)}$ & \textbf{0.44850}$^2$ & 0.65160  & 0.61698 & 0.62549 & $\underline{(13,7)}$ & \textbf{0.52902}$^2$ & 0.77148  & 0.75796 & 0.76792\\
$\underline{(14,7)}$ & \textbf{0.60538}$^2$ & 0.88384  & 0.89915 & 0.91027 & $(8,8)$ & \textbf{0.01889}$^1$& 0.05769 & 0.01922 & 0.01923\\
$(9,8)$ & \textbf{0.05616}$^1$ & 0.12874 & 0.06001 & 0.06013 & $\underline{(10,8)}$ & \textbf{0.10791}$^2$  & 0.20754 & 0.12048 & 0.12096 \\
$\underline{(11,8)}$ & \textbf{0.16878}$^2$  & 0.29023  & 0.19680 & 0.19818 & $\underline{(12,8)}$ & \textbf{0.23451}$^2$ & 0.37434  & 0.28470 & 0.28797\\
$\underline{(13,8)}$ & \textbf{0.30214}$^2$  & 0.45827 & 0.38245 & 0.38694 & $\underline{(14,8)}$ & \textbf{0.36974}$^2$  & 0.56612 & 0.48658 & 0.49227\\
$(10,9)$ & \textbf{0.02773}$^1$ & 0.07668 & 0.02874 & 0.02876 & $\underline{(11,9)}$ & \textbf{0.05796}$^2$  & 0.13098 & 0.06197 & 0.06208\\
$\underline{(12,9)}$ & \textbf{0.09730}$^2$  & 0.19157 & 0.10746 & 0.10778 & $\underline{(13,9)}$ & \textbf{0.14332}$^2$  & 0.25611 & 0.16368 & 0.16444\\
$\underline{(14,9)}$ & \textbf{0.19382}$^2$  & 0.32294 & 0.22865 & 0.23033 & $(11,10)$ & \textbf{0.01321}$^1$  & 0.04289 & 0.01342 & 0.01343\\
$(12,10)$ & \textbf{0.02978}$^1$ &  0.07806 & 0.03093  & 0.03095 & $\underline{(13,10)}$ & \textbf{0.05342}$^2$ & 0.12009 & 0.05674 & 0.05681\\
$\underline{(14,10)}$ & \textbf{0.08332}$^2$ & 0.16726  & 0.09071 & 0.09090 & $(13, 11)$ & \textbf{0.01476}$^1$ &  0.04400 & 0.01506  & 0.01506\\
$\underline{(14, 11)}$ & \textbf{0.02815}$^2$ & 0.07141 & 0.02915 & 0.02916 & $(14,12)$ & \textbf{0.00712}$^1$ &  0.02361 & 0.00718 & 0.00718\\
$(15,13)$ & \textbf{0.00335}$^1$ &  0.01218 & 0.00336 & 0.00336\\
\hline
\vspace{-2mm}
\end{tabularx}
\label{tab:bkbounds}
\raggedright
\footnotesize{$^1$Bounds obtained with the procedure of Section \ref{ClusterBound}, strictly improving also the generalization of \cite{costaDalai} to the $(b,k)$-case.\\ $^2$Bounds where the procedure of Section \ref{ClusterBound} reduces to the same solution obtained by generalization of \cite{costaDalai}.}
\end{table}

\begin{table}[ht!]
\footnotesize
\def\arraystretch{1.17}
\caption{Upper bounds on $R_{(b,k)}$. All numbers are rounded upwards.}
\centering
\begin{tabularx}{\linewidth}{c@{\extracolsep{\fill}}c@{\extracolsep{\fill}}c@{\extracolsep{\fill}}c@{\extracolsep{\fill}}c@{\extracolsep{\fill}}c}
$(b,k)$ & \cite{DalaiVenkatJaikumar}* & \cite{costaDalai}* & \cite{Arikan} & \cite{venkat} & \cite{Korner2} \\
\hline
$(5,4)$ &\textbf{0.57303} & 0.66126 & 0.61142 & 0.74834  & 0.73697(0)\\
$(6,4)$ & \textbf{0.77709} & 0.87963  & 0.83904 & 1.09604 & 1.00000(0)\\
$(7,4)$ & \textbf{0.94372} &  1.03711 & 1.02931 & 1.40593  &  1.22239(0) \\
$(100, 6)$ & \textbf{2.81342} & --- & 3.61848(2) & 4.87959(2)  &  4.32193(0) \\
$(100, 7)$ & \textbf{2.67473} & --- & 3.41158(2) & 4.47696(2)  & 4.05889(0) \\
\hline
\vspace{-2mm}
\end{tabularx}
\label{tab:dalaivenkatjaikumar}
\raggedright
\footnotesize{Missing values indicate impossibility to compute the bound due to high computational complexity.\\$^*$The generalized bound for the $(b,k)$ case.}
\end{table}

\begin{table}[ht!]
\footnotesize
\def\arraystretch{1.17}
\caption{Upper bounds on $R_{(b,k)}$. All numbers are rounded upwards.}
\centering
\begin{tabularx}{\linewidth}{c@{\extracolsep{\fill}}c@{\extracolsep{\fill}}c@{\extracolsep{\fill}}c@{\extracolsep{\fill}}c@{\extracolsep{\fill}}c@{\extracolsep{\fill}}c}
$(b,k)$ & \cite{venkat} & \cite{costaDalai}* & \cite{Arikan} & \cite{Korner2} \\
\hline
$(9,9)$ &\textbf{8.4288} $\cdot \mathbf{10^{-3}}$ & 0.00946 & 0.03182  & $8.4300 \cdot 10^{-3}$\\
$(10,10)$ & \textbf{3.6287} $\cdot \mathbf{10^{-3}}$ &  0.00419 & 0.01642 &  $3.6288 \cdot 10^{-3}$\\
$(11, 11)$ & \textbf{1.53895} $\cdot \mathbf{10^{-3}}$ &  0.00181 & 0.00803  &  $1.53897 \cdot 10^{-3}$ \\
$(12, 11)$ & \textbf{6.13036} $\cdot \mathbf{10^{-3}}$ &  0.00664 & 0.02266  &  $6.13075 \cdot 10^{-3}$\\
$(12, 12)$ & \textbf{6.44678} $\cdot \mathbf{10^{-4}}$ & 0.00077 & 0.00377  &  $6.44679  \cdot 10^{-4}$ \\
$(13, 12)$ & \textbf{2.75350} $\cdot \mathbf{10^{-3}}$ &  0.00305 & 0.01143  &  $2.75355  \cdot 10^{-3}$ \\
$(13, 13)$ & \textbf{2.672760} $\cdot \mathbf{10^{-4}}$ &  0.00033 & 0.00172 &  $2.672761  \cdot 10^{-4}$\\
$(14, 13)$ & \textbf{1.218595} $\cdot \mathbf{10^{-3}}$ &  0.00138 & 0.00556 &  $1.218599 \cdot 10^{-3}$\\
\hline
\vspace{-2mm}
\end{tabularx}
\label{tab:venkatriazanov}
\raggedright
\footnotesize{$^*$The generalized bound for the $(b,k)$ case.}
\end{table}

The paper is structured as follows. In Section \ref{background} we give some background proving Lemmas \ref{dalaivenkat_lemma} and \ref{Rbk_Mj_lemma}. In Section \ref{ClusterBound} we present the first part of our computational method to bound $\mathbf{M}_j$ by partitioning the domain of possible $p$ and $q$ distributions and then working on the subdomains. The second part is presented in Section \ref{CaseReduction}, where we derive optimality conditions on $p$ and $q$ over such subdomains, which allow us to reduce the problem to a manageable one that can be solved exactly. Finally, in Section \ref{sec:qualk6} we show that at least some of the bounds that we obtain are not tight, although a quantitative improvement is not explicitly derived.


\section{Background}
\label{background}
The best upper bounds on $R_{(b,k)}$ available in the literature can all be seen as different applications of a central idea, which is the study of $(b,k)$-hashing by comparison with a combination of binary partitions. This mainline of approach to the problem comes from the original work of Fredman and Koml\'os \cite{FredmanKomlos}. A clear and productive formulation of the idea was given by Radhakrishnan in terms of Hansel's lemma \cite{Jkumar}, which remained the main tool used in all recent results \cite{DalaiVenkatJaikumar2}, \cite{venkat} and \cite{costaDalai}. 

A \textit{hypergraph} $\mathcal{H}$ is a family $E$ of subsets of a finite set $V$  where the subsets in $E$ are called edges and the elements of $V$ are called vertices. If all the edges have size $d$ then we say that $\mathcal{H}$ is a $d$-uniform hypergraph.  We state the Hansel's Lemma here for the reader's convenience.

\begin{lemma}[Hansel for Hypergraphs \cite{hansel}, \cite{nilli}] \label{hansel}
Let $K_r^d$ be the complete $d$-uniform hypergraph on $r$ vertices and let $G_1,\ldots,G_m$ be $c$-partite $d$-uniform hypergraphs on those vertices such that $\cup_{i}G_i=K_r^d$. Let $\tau(G_i)$ be the number of non-isolated vertices in $G_i$. Then
\begin{equation}
\log \frac{c}{d-1} \sum_{i=1}^m \tau(G_i) \geq \log \frac{r}{d-1} \,.
\end{equation}
\end{lemma}
Using this main ingredient, we provide here a proof of Lemma \ref{Rbk_Mj_lemma}, which extends the bound used in \cite{costaDalai}  to general $b$ and $k$. We refer the reader to \cite{costaDalai} for a more detailed discussion on connections with other previous bounds in the literature. 

\begin{proof}[Proof of Lemma \ref{Rbk_Mj_lemma}]

Given a $(b,k)$-hash code $C$ of rate $R$, fix any $j$ elements $x_1,x_2,\ldots, x_{j}$ in $C$, with $j$ in the range $2, \ldots, k-2$. For any coordinate $i$ let $G_i^{x_1, \ldots, x_j}$ be the $(b-j)$-partite $(k-j)$-uniform hypergraph with vertex set $C\setminus\{x_1,x_2,\ldots, x_j\}$ and edge set
\begin{equation}\label{eq:FKgraph}
E = \big\{ \left\{y_1, \ldots, y_{k-j}\right\} : x_{1,i}, \ldots, x_{j,i},y_{1,i}, \ldots,y_{k-j, i} \mbox{ are all distinct} \big\}\,.
\end{equation}
Since $C$ is a $(b,k)$-hash code, then $\bigcup_i G_i^{x_1, \ldots, x_{j}}$ is the complete $(k-j)$-uniform hypergraph on $C\setminus\{x_1,x_2,\ldots, x_j\}$ and so
\begin{equation}
\log \frac{b-j}{k-j-1} \sum_{i=1}^n \tau(G_i^{x_1, \ldots, x_{j}})\geq \log\frac{|C|-j}{k-j-1}\,.
\label{eq:hansel_hash}
\end{equation}
Inequality (\ref{eq:hansel_hash}) holds for any choice of $x_1,x_2,\ldots, x_j$, so the main goal is proving that the left hand side is not too large for all possible choices of $x_1,x_2,\ldots, x_j$. The choice can be deterministic or we can take the expectation over any random selection.

First note that if the $x_{1,i},x_{2,i},\ldots, x_{j,i}$ are not all distinct (let us say that they ``collide'') then the hypergraph defined by \eqref{eq:FKgraph} is empty, that is the corresponding $\tau$ in the left hand side of \eqref{eq:hansel_hash} is zero. Otherwise, $\tau(G_i^{x_1, \ldots, x_j})$ depends on the frequency of different symbols in the $i$-th coordinate of the code.
Let  $f_{i}$ be their distribution, meaning that $f_{i,a}$ is the fraction of elements of $C$ whose $i$-th coordinate is $a$. Then, we have 
\begin{multline}
\tau(G_i^{x_1, \ldots, x_j}) =
\begin{cases}
0 \hspace{4.5cm} x_1,\ldots,x_{j} \mbox{ collide in coordinate }i\\
\left(\frac{|C|}{|C|-j}\right)\left(1-\sum_{h=1}^{j}f_{i,x_{h,i}}\right) \hspace{0.5cm}\mbox{otherwise}
\end{cases}.
\end{multline}


We partition the code $C$ into subcodes $C_\omega$, $\omega\in\Omega$ in such a way that each subcode has a size which grows unbounded with $n$ and uses in any of its first $\ell$ coordinates only $j-1$ symbols, where $\ell$ denotes the length of the prefix. It can be shown, by an easy extension of the method used for the case $b=k$ and $j=k-2$ in  \cite{costaDalai}, that if the original code has rate $R$, then for any $\epsilon>0$ one can do this with a choice of $\ell=n(R-\epsilon)/\log\left(\frac{b}{j-1}\right)$ for $n$ large enough. Given such a partition of our code, if we select codewords $x_{1}, \ldots,x_{j}$ within the same subcode $C_\omega$, they will collide in the first $\ell$ coordinates and the corresponding contribution to the left-hand side of \eqref{eq:hansel_hash} will be zero. The next step is to add randomization. Pick randomly one of the subcodes $C_\omega$ and randomly select the codewords $x_{1}, \ldots,x_{j}$ within $C_\omega$. Then an upper bound on $|C|$ is obtained by taking an expectation on the left-hand side of \eqref{eq:hansel_hash} 
\begin{align}
\log \frac{|C|-j}{k-j-1} &\leq \log \frac{b-j}{k-j-1}  \mathbb{E}_{\omega}\left(\mathbb{E}\left[\sum_{i\in [\ell+1,n]}\tau(G_i^{x_1,x_2,\dots,x_{j}})|\omega\right]\right)  \nonumber \\
&= \log \frac{b-j}{k-j-1}  \sum_{i\in [\ell+1,n]}\mathbb{E}_{\omega}(\mathbb{E}[\tau(G_i^{x_1,x_2,\dots,x_{j}})|\omega])\label{eq:sumellton}.
\end{align}
Here, each subcode $C_\omega$ is taken with probability $\lambda_{\omega}=|C_\omega|/|C|$, and $x_{1}, \ldots,x_{j}$ are taken uniformly at random (without repetitions) from $C_\omega$.

Let now $f_{i|\omega}$ be the distribution of the $i$-th coordinate of the subcode $C_\omega$ (with components, say, $f_{i,a|\omega}$) . Then, for $i>\ell$, we can write
\begin{align}
\mathbb{E} [& \tau(G_i^{x_1, \ldots, x_{j}})|\omega] =\left(1+o(1)\right) \sum_{\stackrel{\text{distinct }}{ a_1,\ldots,a_{j}}} f_{i,a_1|\omega}f_{i,a_2|\omega}\cdots f_{i,a_{j}|\omega}(1-f_{i,a_1}-\ldots-f_{i,a_{j}} )\label{eq:exptaugf}
\end{align}
where the $o(1)$ is meant as $n\to\infty$ and is due, under the assumption that $C_\omega$ grows unbounded with $n$, to sampling without replacement within $C_\omega$.
Now, since $\lambda_{\omega}=|C_\omega|/|C|$, $f_i$ is actually the expectation of $f_{i|\omega}$ over $\omega$, that is, using a different dummy variable $\mu$ to index the subcodes for convenience,
$$
f_i=\sum_{\mu}\lambda_{\mu}f_{i|\mu}\,.
$$
Using this in \eqref{eq:exptaugf}, one notices that when taking a further expectation over $\omega$ it is possible to operate a symmetrization in $\omega$ and $\mu$. 
The expectation of \eqref{eq:exptaugf} over $\omega$ can then be written as
\begin{align}
\mathbb{E}_{\omega}[\tau (G_i^{x_1,x_2,\dots,x_{j}})] =\left(1+o(1)\right)\frac{1}{2}\sum_{\omega,\mu\in \Omega}\lambda_{\omega}\lambda_{\mu}\Psi_j(f_{i|\omega},f_{i|\mu})\,,\label{simmetrizzata2}
\end{align}
so that
\begin{equation}
\mathbb{E}_{\omega}[\tau (G_i^{x_1,x_2,\dots,x_{j}})]\leq (1+o(1))\frac{1}{2}\mathbf{M}_j\,.
\end{equation}
This leads to
\begin{equation}
\log{|C|} \leq (1+o(1)) \frac{1}{2} (n-\ell) \mathbf{M}_j \log \frac{b-j}{k-j-1} \,,
\end{equation}
from which, using the value of $\ell$ described above, one deduces
\begin{align*}
R\leq (1+o(1))\frac{1}{2}\left[1-\frac{R}{\log\left(\frac{b}{j-1}\right)}\right]\mathbf{M}_j \log \frac{b-j}{k-j-1}.
\end{align*}
Explicitating in $R$ we conclude the proof of the Lemma.
\end{proof}

The first part of the above derivation follows the same method used in \cite{DalaiVenkatJaikumar}. In particular, the proof of Lemma \ref{dalaivenkat_lemma} can be obtained using $j=2$ and looking at \eqref{eq:exptaugf} as a quadratic form in $f_{i|\omega}$ with kernel of elements $(1-f_{i,a_1}-f_{i,a_2})$. The procedure used in \cite{DalaiVenkatJaikumar} can then be applied also for $b \geq k$ with some simple variations. 

\begin{proof}[Proof of Lemma \ref{dalaivenkat_lemma}]
Set $j=2$ in \eqref{eq:exptaugf}. Proceeding as in \cite{DalaiVenkatJaikumar}, it can be shown that the right hand side, as a quadratic form in $f_{i|\omega}$, is a concave function on the simplex of probability distributions if all the values $f_{i,a}$ are not larger than $1/2$. Assume first that this holds for all $i\in[\ell+1,n]$. The expectation over $\omega$ is then bounded by the value obtained by replacing both $f_{i|\omega}$ and $f_i$ with a uniform distribution, which is easily evaluated to be $(b^2-3b+2)/b^2$. When used in \eqref{eq:sumellton} this gives the bound of Lemma \ref{dalaivenkat_lemma}. It remains to show that we can assume without loss of generality that $f_{i,a}\leq 1/2$ for all $i$ and $a$. Again the procedure is a generalization of what was done in \cite{DalaiVenkatJaikumar}. Suppose that there exists a coordinate $i \in \{1, 2, \ldots, n\}$ for which (rename the symbols if needed) $f_{i,1}\geq f_{i,2}\geq\ldots\geq f_{i,b}$ with $f_{i,1} > 1/2$. Note that we must then have $f_{i,1}+f_{i,2}+\ldots+f_{i,k-1}\geq (b+k-3)/(2b-2)$.
 We can build another $(b,k)$-hash  code $C'$ by removing all the codewords in $C$ for which the symbol in the $i$-th coordinate is in $\{k,k+1,\ldots,b\}$ and by deleting this coordinate in the remaining codewords. Clearly $C'$ has length $n-1$ and cardinality $|C'|\geq |C|\cdot(b+k-3)/(2b-2)$. This process can be iterated, say $t$ times, in order to get a code $\tilde{C}$ of length $n-t$ in which $f_{i,a}\leq   1/2$ for all $i \in \{1, 2, \ldots, n - t\}$ and for all $a \in \{1,2, \ldots, b\}$ and such that 
\begin{equation}
|\tilde{C}|\geq |C|\left(\frac{b+k-3}{2b-2}\right)^t\,.
\end{equation}
Let $B(b,k)$ be the right hand side of \eqref{eq:dalaivenkatbound}.
We can apply the previous part of the proof to $\tilde{C}$ and bound the rate $R$ of $C$ as 
\begin{align*}
\frac{1}{n}\log|C|
& \leq \frac{1}{n}\log|\tilde{C}|+\frac{t}{n}\log\left(\frac{2b-2}{b+k-3}\right) \\
& \leq \frac{n-t}{n}B(b,k)+\frac{t}{n}\log\left(\frac{2b-2}{b+k-3}\right) +o(1)\\
& \leq B(b,k)-\frac{t}{n}\left[B(b,k)-\log\left(\frac{2b-2}{b+k-3}\right)\right]+o(1)\,.
\end{align*}
The proof of the Lemma is concluded if we prove that $B(b,k) > \log\frac{2b-2}{b+k-3}$ for $b \geq k \geq 4$. 
We verify this inequality considering the following three different ranges of $b$ and $k$:
\begin{enumerate}
\item Suppose that $12 \leq k \leq b\leq (k-3)^2$.  Then
\begin{align*}
	B(b,k)  
	&\overset{(i)}{>} \frac{2}{3} \cdot \frac{(b^2-3b+2) \log b \log\left(\frac{b-2}{k-3} \right)}{ b^2 \log(b) } > \frac{2}{3} (1-3/b) \log\left(\frac{b-2}{k-3}\right)  \\ 
    & \overset{(ii)}{\geq} \frac{1}{2} \log\left(\frac{b-2}{k-3}\right) \,,
\end{align*} 
where $(i)$ is true since $\log\left(\frac{b-2}{k-3}\right) \leq 1/2 \log b$ for $b \leq (k-3)^2$,  while $(ii)$ since $b \geq 12$.
Then, it can be verified that for $b \geq k \geq 12$ we have that
$$
	\frac{1}{2} \log\left(\frac{b-2}{k-3}\right) > \log\left(\frac{2b-2}{b+k-3}\right)\,.
$$
\item Suppose that $b \geq 8 k -22$ and $k \geq 4$.  Then

\begin{align*}
	B(b,k)
	&>
	 \frac{(b^2-3b+2) \log b \log\left(\frac{b-2}{k-3} \right)}{ 2 b^2 \log(b) } > \frac{1}{2} (1-3/b) \log\left(\frac{b-2}{k-3} \right)  \nonumber \\
	 &\overset{(i)}{>} \frac{1}{3} \log\left(\frac{b-2}{k-3} \right) \,, 
\end{align*}
where $(i)$ is true since $b > 9$.  Then, it can be easily verified that for $b \geq 8k-22$ we have that
$$
	\frac{1}{3} \log\left(\frac{b-2}{k-3} \right) \geq 1 > \log\left(\frac{2b-2}{b+k-3}\right)\,.
$$
\item All the cases  $b \geq k = 4, 5, \ldots, 11$ can be verified manually or by using a symbolic computation software.
\end{enumerate}
Finally,  we see that the ranges of $b$, as functions of $k$, in the first two cases intersect because
$$
	(k-3)^2 \geq 8k-22
$$
is verified for every $k \geq 12$. Therefore the thesis of the lemma follows.
\end{proof}

\section{Bounding the quadratic form}
\label{ClusterBound}
We now enter the problem of determining better upper bounds on the value of $\mathbf{M}_j$ defined in \eqref{eq:QuadraticForm}. 
We consider partitions of $\mathcal{P}_b$, the set of probability distributions on $b$ elements, into disjoint subsets to find upper bounds on the quadratic form \eqref{eq:QuadraticForm} in terms of simpler ones. If we have a partition $\{\mathcal{P}_b^0,\mathcal{P}_b^1,\ldots,\mathcal{P}_b^r\}$ of $\mathcal{P}_b$ and we define
$$
m_{i,h}=\sup_{p\in \mathcal{P}_b^i , q\in\mathcal{P}_b^h} \Psi_j(p,q)\,,\qquad \eta_i=\sum_{p\in\mathcal{P}_b^i}\lambda_p\,,
$$
then clearly 
\begin{align}
\sum_{p,q}\lambda_p\lambda_q \Psi_j(p,q) & \leq \sum_{i,h}\sum_{p\in\mathcal{P}_b^i}\sum_{q\in\mathcal{P}_b^h}\lambda_p \lambda_q m_{i,h} \leq \sum_{i,h}\eta_i \eta_h m_{i,h}\,.\label{eq:ReducedQuadratic}
\end{align}
This is a convenient simplification since we have now an $r$-dimensional problem which we might be able to deal with in some computationally feasible way. We will use this procedure with two different partitions in terms of how balanced or unbalanced the distributions are.
We take $b+1$ subsets with some symmetry which allows us to further reduce the complexity.

\textbf{Partition based on maximum value.} We first consider a partition of $\mathcal{P}_b$ in terms of the largest probability value which appears in a distribution. We use a parameter $\epsilon \leq 1/(j+1)$; all quantities will depend on $\epsilon$ but we do not write this to avoid cluttering the notation. We define $b$ sets of unbalanced distributions
$$
\widecheck{\mathcal{P}}_b^{i} = \left\{p\in \mathcal{P}_b:p_i>1-\epsilon\right\}\,
$$
for every $1\leq i \leq b$, and correspondingly a set of balanced distributions
$$
 \widecheck{\mathcal{P}}_b^{0} = \left\{p\in \mathcal{P}_b:p_i\leq 1-\epsilon\ \forall i\right\}\,.
$$
Note that these are all disjoint sets since $\epsilon<1/2$ when $j \geq 2$. Following the scheme mentioned above, we can consider the values $m_{i,h}$ and $\eta_i$ for this specific partition. However, due to symmetry, the values $m_{i,h}$ can be reduced to only four cases, depending on whether $p$ and $q$ are both balanced, one balanced and one unbalanced, or both unbalanced, either on the same coordinate or on different coordinates.\\
Assuming $1\leq i,h\leq b$ with $i\neq h$, the following quantities are then well defined and independent of the specific values chosen for $i$ and $h$
\begin{equation}
\begin{aligned}
\widecheck{M}_1 & =\sup_{p,q \in \widecheck{\mathcal{P}}_b^{0}} \Psi_j(p;q) &\qquad &
\widecheck{M}_2 & =\sup_{p \in \widecheck{\mathcal{P}}_b^{0},q\in \widecheck{\mathcal{P}}_b^{i}} \Psi_j(p;q)\\
\widecheck{M}_3 & =\sup_{p,q \in \widecheck{\mathcal{P}}_b^{i}} \Psi_j(p;q) &\qquad &
\widecheck{M}_4 & =\sup_{p \in \widecheck{\mathcal{P}}_b^{i},q\in \widecheck{\mathcal{P}}_b^{h}} \Psi_j(p;q)
\end{aligned}
\label{eq:MaxMis}
\end{equation}
These values can then be used in \eqref{eq:ReducedQuadratic} in place of the values $m_{i,h}$.

\textbf{Partition based on the minimum value.} We also consider a partition of $\mathcal{P}_b$ using constraints from below. Again we use a parameter $\epsilon$ which will be then tuned. We assume here $\epsilon<1/b$. Consider now the following disjoint sets of unbalanced distributions
$$
\widehat{\mathcal{P}}_b^{i} = \left\{p\in \mathcal{P}_b:p_i<\epsilon\,,p_h\geq p_i\ \forall h\,,p_h>p_i\ \forall h<i\right\}\,
$$
for $1 \leq i \leq b$, that is, distributions in $\widehat{\mathcal{P}}_b^{i}$ have a minimum component in the $i$-th coordinate, which is smaller than $\epsilon$, and strictly smaller than any of the preceding components (unless of course $i=1$). Correspondingly, define a set of balanced distributions as
$$
 \widehat{\mathcal{P}}_b^{0} = \left\{p\in \mathcal{P}_b:p_i\geq \epsilon\ \forall i\right\}\,.
$$
The symmetry argument mentioned before also applies in this case and we can continue in analogy replacing the $m_{i,h}$ of \eqref{eq:ReducedQuadratic} with the following quantities
\begin{equation}
\begin{aligned}
\widehat{M}_1 & =\sup_{p,q \in \widehat{\mathcal{P}}_b^{0}} \Psi_j(p;q) &\qquad &
\widehat{M}_2 & =\sup_{p \in \widehat{\mathcal{P}}_b^{0},q\in \widehat{\mathcal{P}}_b^{i}} \Psi_j(p;q)\\
\widehat{M}_3 & =\sup_{p,q \in \widehat{\mathcal{P}}_b^{i}} \Psi_j(p;q) & \qquad &
\widehat{M}_4 & =\sup_{p \in \widehat{\mathcal{P}}_b^{i},q\in \widehat{\mathcal{P}}_b^{h}} \Psi_j(p;q)
\end{aligned}
\label{eq:MinMis}
\end{equation}
where again $1\leq i,h\leq b$ with $i\neq h$. 

\textbf{Quadratic form.} Applying the above scheme with the symmetric partitions we just defined, we can now rewrite the upper bound of equation \eqref{eq:ReducedQuadratic} in the form
\begin{align}
\sum_{p,q}\lambda_p\lambda_q \Psi_j(p;q) & \leq \eta_0^2M_1+2\eta_0\sum_{i=1}^b \eta_i M_2+\sum_{i=1}^b \eta_i^2M_3+ 2\sum_{i<h} \eta_i\eta_h M_4\,,
\label{eq:QuadracBoundedbyMs}
\end{align}
where either the $\widehat{M}_i$'s or the $\widecheck{M}_i$'s can be used in place of the $M_i$'s.

Call $M$ the maximum value achieved by the right hand side of \eqref{eq:QuadracBoundedbyMs} over all possible probability distributions $\eta=(\eta_0,\eta_1,\ldots,\eta_b)$. We show that under assumptions that are verified in our setting, the value of $M$ can be determined explicitly.

\begin{lemma}
\label{lemmaMax}
Let $M_1, M_2, M_3$ and $M_4$ be positive numbers such that $M_4 > M_3$ and, for a probability distribution $\eta=(\eta_0,\eta_1,\ldots,\eta_b)$, define the function
$$
f(\eta)=\eta_0^2M_1+2\eta_0\sum_{i=1}^b \eta_i M_2+\sum_{i=1}^b \eta_i^2M_3+ 2\sum_{i<h} \eta_i\eta_h M_4\,.
$$
Then
\begin{equation}
M=\max_{\eta}f(\eta)
\label{eq:defM}
\end{equation}
is attained at $\eta_1=\eta_2=\ldots =\eta_b$ and
$$\eta_0 = \begin{cases}
    \frac{M_2- \frac{1}{b} M_3- \frac{b-1}{b} M_4}{2M_2-M_1- \frac{1}{b} M_3- \frac{b-1}{b} M_4},& \text{if } M_2 > M_1{,}M_3{,}M_4\\
    0 \text{ or } 1,              & \text{otherwise} \\
    \end{cases}.
$$
\end{lemma}
\proof
Since $\sum_{i=1}^b \eta_i=(1-\eta_0)$, $f$ can be written as
$$\eta_0^2M_1+2(1-\eta_0)\eta_0 M_2+\sum_{i=1}^b \eta_i^2M_3+ 2\sum_{i<h} \eta_i\eta_h M_4.$$
Note that
$$\sum_{i=1}^b \eta_i^2M_3+2\sum_{i<h} \eta_i\eta_h M_4=\sum_{i=1}^b \eta_i^2(M_3-M_4)+(1-\eta_0)^2 M_4.$$
Since $M_3 < M_4$ and $\sum_{i=1}^b\eta_i=1-\eta_0$, this sum is maximized when $\eta_1=\eta_2=\ldots=\eta_b=(1-\eta_0)/b$. Therefore we have to maximize the quantity
$$
\eta_0^2M_1+2(1-\eta_0)\eta_0 M_2+\frac{1}{b}(1-\eta_0)^2(M_3-M_4)+ (1-\eta_0)^2 M_4\,,
$$
which is just a quadratic in $\eta_0$ that achieves its maximum in $[0,1]$ at the point described in the statement of the Lemma.
\endproof

We will describe in the next Section our procedure to determine, or upper bound the values $\widehat{M}_i$, $\widecheck{M}_i$. Using these bounds in equation \eqref{eq:QuadracBoundedbyMs} we thus obtain an upper bound on $\mathbf{M}_j$ defined in \eqref{eq:QuadraticForm}. Applying Lemma \ref{Rbk_Mj_lemma} we obtain our main result.

\begin{teo}
The bounds of Table \ref{tab:bkbounds} hold.
\label{th:bounds}
\end{teo}

\begin{rem}\label{rem:ratesunderlined}
The bounds on $R_{(7, 7)}$, $R_{(8, 8)}$, $R_{(9,8)}$, $R_{(10, 9)}$, $R_{(11,10)}$, $R_{(12, 10)}$, $R_{(13, 11)}$,  $R_{(14,12)}$ and  $R_{(15, 13)}$ are obtained using the partition based on the maximum value $\{\widecheck{\mathcal{P}}_b^{i}\}_{i=0,\ldots,b}$.
The bounds on $R_{(5,5)}$, $R_{(6, 5)}$ and $R_{(6,6)}$ are obtained using the partition based on the minimum value $\{\widehat{\mathcal{P}}_b^{i}\}_{i=0,\ldots,b}$.

All other cases, those underlined in Table \ref{tab:bkbounds}, are obtained computing, as done in \cite{costaDalai}, the global maximum of $\Psi_{k-2}$, which is attained for uniform distributions. Therefore, the partitioning process in these particular cases cannot make any improvements.
\end{rem}

Based on the result in \cite{DalaiVenkatJaikumar2}, or its generalization given in equation \eqref{eq:dalaivenkatbound} and on Theorem \ref{th:bounds} for $(b,k)=(6,6)$, we are led to formulate the following conjecture.	
\begin{con}\label{conj1}
For $b \geq k>3$, 
$$
	R_{(b,k)} \leq \min_{2 \leq j \leq k-2} \left(\frac{1}{\log\frac{b}{j-1}} + \frac{b^{j+1}}{b^{\underline{j+1}} \log \frac{b-j}{k-j-1}} \right)^{-1}\,.
$$
\end{con}
Note that the conjectured expression can be seen as a modification of the K\"orner-Marton bound in \eqref{eq:kornermarton} which takes into account the effects of prefix-based partitions.

\section{Computation of \texorpdfstring{$M$}{Lg} in \eqref{eq:defM}}
\label{CaseReduction}

In light of Lemma \ref{lemmaMax}, the main problem for the computation of $M$ is determining the $\widecheck{M}_i$'s and $\widehat{M}_i$'s defined in equations \eqref{eq:MaxMis} and \eqref{eq:MinMis}. This requires determining the maximum values taken by $\Psi_j(p;q)$ for $p$ and $q$ constrained to specific subsets $\widecheck{\mathcal{P}}_b^{i}$ or $\widehat{\mathcal{P}}_b^{i} $. Following a procedure similar to that of \cite{costaDalai}, here we prove that, under certain conditions, the distributions $p$ and $q$ achieving those maxima have many equal components. This, together with other simplifications that will be presented later, allows us to greatly reduce the complexity in the search for the maxima (see Remarks \ref{rem:4vars} and  \ref{rem:3vars} below). For this purpose we first present three Lemmas, which generalize Lemmas 3, 4 and 5 of \cite{costaDalai}.
\begin{lemma}[Extension of Lemma 3 in \cite{costaDalai}]
\label{lemma1}
Let $\ell$ be an integer in $[2,b]$ and, for $i\in [1,\ell]$, consider the nonempty intervals $I_i=[a_i,b_i]$ and $J_i=[c_i,d_i]$.
Set $D_p=I_1\times I_2\times \cdots \times I_{\ell}\times \overline{p_{\ell+1}}\times\cdots \times \overline{p_{b}}$ and $D_q=J_1\times J_2\times \cdots \times J_{\ell}\times \overline{q_{\ell+1}}\times\cdots \times \overline{q_{b}}$. Consider the set $D$ of pairs of probability vectors $(p,q)$ such that $p$ belongs to $D_p$ and $q$ belongs to $D_q$.
Then if $(\overline{p}; \overline{q})$ is a maximum point for $\Psi_j$ in $D$ then either $\overline{p_i}=\overline{p_h}$ and $\overline{q_i}=\overline{q_h}$ for any $i,h\in [1,\ell]$ or there is a maximum for $\Psi_j$ on the boundary of $D$ (as projected on the first $\ell$ coordinates).
\end{lemma}
Note that, in particular, in the latter case, we have a maximum point $(\overline{p}; \overline{q})$ for $\Psi_j$ with at least one index $i\in [1,\ell]$ such that either $\overline{p_i}\in \{a_i,b_i\}$ or $\overline{q_i}\in \{c_i,d_i\}$.
\proof
Let us assume that $\overline{P}=(\overline{p};\overline{q})$ is a maximum point for $\Psi_j$ in $D$ and $\overline{p_1},\overline{p_2},\dots,\overline{p_{\ell}}$ or $\overline{q_1},\overline{q_2},\dots,\overline{q_{\ell}}$ are not all equal. By symmetry, assume without loss of generality that $\overline{p_1}\not=\overline{p_2}$.
Now, if $\overline{P}$ is a maximum for $\Psi_j$ not on the boundary $D$, then it is a maximum also under the stronger constraints $p_1+p_2=c_1$, $q_1+q_2=c_2$ where $c_1=\overline{p_1}+\overline{p_1}$, $c_2=\overline{q_1}+\overline{q_2}$, and $p_i=\overline{p_i},q_i=\overline{q_i}$ for $i\in\{3,4,\dots, \ell\}$.
Then, let us consider the line $L$ of points $P(t)$ such that
$$P(t)=P(0)+t\left(\frac{\overline{p_1}-\overline{p_2}}{2},\frac{-\overline{p_1}+\overline{p_2}}{2},0,\dots,0;\frac{\overline{q_1}-\overline{q_2}}{2},\frac{-\overline{q_1}+\overline{q_2}}{2},0,\dots,0\right)\,,$$
where $P(0)=(\frac{\overline{p}_1+\overline{p}_2}{2},\frac{\overline{p}_1+\overline{p}_2}{2},\overline{p}_3,\dots,\overline{p}_b;$ $\frac{\overline{q}_1+\overline{q}_2}{2},\frac{\overline{q}_1+\overline{q}_2}{2},\overline{q}_3,\dots,\overline{q}_b)$, so that $P(1)=\bar{P}$.

It is easy to see that $\Psi_j(P(t))$ is of degree $2$ and, if $\overline{P}$ is not on the boundary of $D$, then, $t=1$ must be a stationary point for $\Psi_j(P(t))$. Moreover $\Psi_j(P(t))$ is an even function because:
\begin{align*}
\Psi_j(P(-t)) & =P(0)-t\left(\frac{\overline{p_1}-\overline{p_2}}{2},\frac{-\overline{p_1}+\overline{p_2}}{2},0,\dots,0;\frac{\overline{q_1}-\overline{q_2}}{2},\frac{-\overline{q_1}+\overline{q_2}}{2},0,\dots,0\right)\\
& = 
P(0)+t\left(\frac{\overline{p_2}-\overline{p_1}}{2},\frac{-\overline{p_2}+\overline{p_1}}{2},0,\dots,0;\frac{\overline{q_2}-\overline{q_1}}{2},\frac{-\overline{q_2}+\overline{q_1}}{2},0,\dots,0\right)\\
& = \Psi_j(P(t)).
\end{align*}
This means that $\Psi_j(P(t))=\alpha t^2+\beta$ for some $\alpha$ and $\beta$ in $\mathbb{R}$. Therefore $t=0$ would be another stationary point for $\Psi_j(P(t))$ but this is possible only if $\alpha=0$ that is $\Psi_j(P(t))$ is a constant.

The thesis follows because, in this case, the maximum is also attained on the boundary of $D$.
\endproof

With essentially the same proof we obtain

\begin{lemma}[Extension of Lemma 4 in \cite{costaDalai}]
\label{lemma2}
Let $\ell$ be an integer in $[2,b]$ and, for $i\in [1,\ell]$, consider the nonempty intervals $I_i=[a_i,b_i]$.
Set $D_p=I_1\times I_2\times \cdots \times I_{\ell}\times \overline{p_{\ell+1}}\times\cdots \times \overline{p_{b}}$ and $D_q=\overline{q_{1}}\times \overline{q_{2}}\times \cdots \times \overline{q_{\ell}}\times \overline{q_{\ell+1}}\times\cdots \times \overline{q_{b}}$ where $\overline{q_i}=\overline{q_h}$ for any $i,h \in [1,\ell]$.
Consider the set $D$ of pairs of probability vectors $(p,q)$ such that $p$ belongs to $D_p$ and $q$ belongs to $D_q$.
Then if $(\overline{p}; \overline{q})$ is a maximum point for $\Psi_j$ in $D$ then either $\overline{p_i}=\overline{p_h}$ for any $i\in [1,\ell]$ or there is a maximum for $\Psi_j$ on the boundary of $D$.
\end{lemma}
Note that, in particular, in the latter case, we have a maximum point $(\overline{p}; \overline{q})$ for $\Psi_j$ with at least one index $i\in [1,\ell]$ such that $\overline{p_i}\in \{a_i,b_i\}$.

Now we present a Lemma that allows us to assume that the coordinates of $p$ and $q$ are properly rearranged depending on their values.
\begin{lemma}[Extension of Lemma 5 in \cite{costaDalai}]
If $p_1 \leq p_2$, and $q_1 \leq q_2$, then
\label{lemma3}
\begin{equation}
\label{lemma3ineq}
\Psi_j(p_1, p_2, p_3, \ldots, p_b; q_1, q_2, q_3,\ldots, q_b) \leq \Psi_j(p_1, p_2, p_3, \ldots, p_b; q_2, q_1, q_3, \ldots, q_b).
\end{equation}
\end{lemma}
\begin{proof}
Using the definition of $\Psi_j$ in eq. \eqref{eq:defPsi}, inequality (\ref{lemma3ineq}) can be restated by only considering the terms in the summation which differ in the two sides, that is, those corresponding to permutations $\sigma$ such that $1 \in \{ \sigma(1), \ldots, \sigma(j) \}$, $\sigma(j+1)=2$ and $2 \in \{ \sigma(1), \ldots, \sigma(j) \}$, $\sigma(j+1)=1$.
Hence inequality (\ref{lemma3ineq}) becomes
\begin{align*}
&(p_1 q_2 + p_2 q_1) \sum_{\sigma \in Sym(3,\ldots, b)} p_{\sigma(3)} \cdots p_{\sigma(j+1)} + q_{\sigma(3)} \cdots q_{\sigma(j+1)} \\
&\leq (p_1 q_1 + p_2 q_2) \sum_{\sigma \in Sym(3,\ldots, b)} p_{\sigma(3)} \cdots p_{\sigma(j+1)} + q_{\sigma(3)} \cdots q_{\sigma(j+1)}
\end{align*}
which can be restated as
\begin{align*}
&(p_2- p_1) (q_2 - q_1) \sum_{\sigma \in Sym(3,\ldots, b)} p_{\sigma(3)} \cdots p_{\sigma(j+1)} + q_{\sigma(3)} \cdots q_{\sigma(j+1)} \geq 0
\end{align*}
This is always true since $p_1 \leq p_2$ and $q_1 \leq q_2$.
\end{proof}
Using the above lemmas, we are able to isolate a relatively small set of possible configurations for the $p$ and $q$ which give the value $\widecheck{M}_1$. 
\begin{prop}\label{reductions}
$\widecheck{M}_1$ is attained in one of the following points:
\begin{itemize}
\item[1)] for $(p;q)$ of the form
$$(\overbrace{0,\dots,0}^{l_1},\underbrace{\alpha,\dots,\alpha}_{l_2},\overbrace{\beta,\dots,\beta}^{b-l_1-l_2-2},\gamma,1-\epsilon;\overbrace{\delta,\dots,\delta}^{l_1},\underbrace{0,\dots,0}_{l_2},\overbrace{\eta,\dots,\eta}^{b-l_1-l_2-2},1-\epsilon,\zeta)$$
where $\alpha,\delta>0,\ \ \beta,\eta,\gamma,\zeta\geq 0 $ and $$l_2\alpha+(b-l_1-l_2-2)\beta+\gamma+(1-\epsilon)=1=l_1\delta+(b-l_1-l_2-2)\eta+(1-\epsilon)+\zeta;$$
\item[2)] for $(p;q)$ of the form
$$(\overbrace{0,\dots,0}^{l_1},\underbrace{\alpha,\dots,\alpha}_{l_2},\overbrace{\beta,\dots,\beta}^{b-l_1-l_2-1},\gamma;\overbrace{\delta,\dots,\delta}^{l_1},\underbrace{0,\dots,0}_{l_2},\overbrace{\eta,\dots,\eta}^{b-l_1-l_2-1},1-\epsilon)$$
where $\alpha,\delta>0,\ \ \beta,\eta,\gamma\geq 0 $ and $$l_2\alpha+(b-l_1-l_2-1)\beta+\gamma=1=l_1\delta+(b-l_1-l_2-1)\eta+(1-\epsilon);$$
\item[3)] for $(p;q)$ of the form
$$(\overbrace{0,\dots,0}^{l_1},\underbrace{\alpha,\dots,\alpha}_{l_2},\overbrace{\beta,\dots,\beta}^{b-l_1-l_2};\overbrace{\delta,\dots,\delta}^{l_1},\underbrace{0,\dots,0}_{l_2},\overbrace{\eta,\dots,\eta}^{b-l_1-l_2})$$
where $\alpha,\delta>0,\ \ \beta,\eta\geq 0 $ and $$l_2\alpha+(b-l_1-l_2)\beta=1=l_1\delta+(b-l_1-l_2)\eta.$$
\end{itemize}
\end{prop}

\begin{proof}
Remember that the value $\widecheck{M}_1$ is the maximum of $\Psi_j$ over pairs $(p,q)$ with $p$ and $q$ in $\widecheck{\mathcal{P}}_b^{0}$.
Moreover, due to Lemma \ref{lemma3}, we have that $p$ and $q$ do not have a value $1-\epsilon$ in the same coordinate.
Similarly, again because of Lemma \ref{lemma3}, either the zeros of $p$ and $q$ are in different positions (i.e. if $p_i=0$ then $q_i\not=0$) or for any $i$ at least one between $p_i$ and $q_i$ is zero.

According to the positions where values $1-\epsilon$ and zero can appear as coordinates of $p$ and $q$, we have that $\widecheck{M}_1$ is attained in one of the following points:
\begin{itemize}
\item[1A)] $p$ and $q$ have respectively $l_1$ and $l_2$ zeros in different positions, both have a coordinate with value $1-\epsilon$ and those are in different positions:
$$(\overbrace{0,\dots,0}^{l_1},\underbrace{\alpha_1,\dots,\alpha_{l_2}}_{l_2},\overbrace{\beta_1,\dots,\beta_{b-l_1-l_2-2}}^{b-l_1-l_2-2},\gamma,1-\epsilon;\overbrace{\delta_1,\dots,\delta_{l_1}}^{l_1},\underbrace{0,\dots,0}_{l_2},\overbrace{\eta_1,\dots,\eta_{b-l_1-l_2-2}}^{b-l_1-l_2-2},1-\epsilon,\zeta);$$
\item[1B)] $p$ and $q$ have respectively $l_1$ and $l_2$ zeros in different positions, additional $b-l_1-l_2-2$ zeros in the same positions, both have a coordinate with value $1-\epsilon$ and those are in different positions:
$$(\overbrace{0,\dots,0}^{l_1},\underbrace{\alpha_1,\dots,\alpha_{l_2}}_{l_2},\overbrace{0,\dots,0}^{b-l_1-l_2-2},0,1-\epsilon;\overbrace{\delta_1,\dots,\delta_{l_1}}^{l_1},\underbrace{0,\dots,0}_{l_2},\overbrace{0,\dots,0}^{b-l_1-l_2-2},1-\epsilon,0);$$
\item[2A)] $p$ and $q$ have respectively $l_1$ and $l_2$ zeros in different positions, $p$ has no coordinate of value $1-\epsilon$ but $q$ has:
$$(\overbrace{0,\dots,0}^{l_1},\underbrace{\alpha_1,\dots,\alpha_{l_2}}_{l_2},\overbrace{\beta_1,\dots,\beta_{b-l_1-l_2-1}}^{b-l_1-l_2-1},\gamma;\overbrace{\delta_1,\dots,\delta_{l_1}}^{l_1},\underbrace{0,\dots,0}_{l_2},\overbrace{\eta_1,\dots,\eta_{b-l_1-l_2-1}}^{b-l_1-l_2-1},1-\epsilon);$$

\item[2B)] $p$ and $q$ have respectively $l_1$ and $l_2$ zeros in different positions, additional $b-l_1-l_2-1$ zeros in the same positions, $p$ has no coordinate of value $1-\epsilon$ but $q$ has:
$$(\overbrace{0,\dots,0}^{l_1},\underbrace{\alpha_1,\dots,\alpha_{l_2}}_{l_2},\overbrace{0,\dots,0}^{b-l_1-l_2-1},0;\overbrace{\delta_1,\dots,\delta_{l_1}}^{l_1},\underbrace{0,\dots,0}_{l_2},\overbrace{0,\dots,0}^{b-l_1-l_2-1},1-\epsilon);$$

\item[3A)] $p$ and $q$ have respectively $l_1$ and $l_2$ zeros in different positions and both have no coordinates with value $1-\epsilon$:
$$(\overbrace{0,\dots,0}^{l_1},\underbrace{\alpha_1,\dots,\alpha_{l_2}}_{l_2},\overbrace{\beta_1,\dots,\beta_{b-l_1-l_2}}^{b-l_1-l_2};\overbrace{\delta_1,\dots,\delta_{l_1}}^{l_1},\underbrace{0,\dots,0}_{l_2},\overbrace{\eta_1,\dots,\eta_{b-l_1-l_2}}^{b-l_1-l_2});$$
\item[3B)] $p$ and $q$ have respectively $l_1$ and $l_2$ zeros in different positions, additional $b-l_1-l_2$ zeros in the same positions and neither has a coordiante of value $1-\epsilon$:
$$(\overbrace{0,\dots,0}^{l_1},\underbrace{\alpha_1,\dots,\alpha_{l_2}}_{l_2},\overbrace{0,\dots,0}^{b-l_1-l_2};\overbrace{\delta_1,\dots,\delta_{l_1}}^{l_1},\underbrace{0,\dots,0}_{l_2},\overbrace{0,\dots,0}^{b-l_1-l_2}).$$
\end{itemize}
Moreover, in all those cases, the allowed domains for $p$ and $q$ satisfy either the hypothesis of Lemma \ref{lemma1} or those of Lemma \ref{lemma2}. This means that we can average the $\alpha$'s (i.e. we can assume that all the $\alpha$'s are equal), the $\beta$'s, the $\delta$'s, and the $\eta$'s. The thesis follows allowing $\beta$ and $\eta$ to possibly be zero and noting that the case $1B$ becomes a subcase of $1A$, $2B$ becomes a subcase of $2A$ and $3B$ becomes a subcase of $3A$.
\end{proof}
\begin{rem}
As seen in Proposition \ref{reductions}, Lemmas \ref{lemma1}, \ref{lemma2} and \ref{lemma3} reduce the maxima candidates to a finite set of possible configurations. Still, the number of such configurations increases with $b$, and the ensuing optimization problems depend on $4$ free variables in the case $1$. The direct evaluation of the maxima of $\Psi_j$ on those configurations can in principle be performed by symbolic computation software, but the resources needed are excessive. In the following lemmas, we provide additional simplifications to obtain the exact evaluations of the maxima.
\label{rem:4vars}
\end{rem}
Due to the following lemma, whose proof can be found in the appendix, we can assume that the number of zeros that appear in $p$ (resp. in $q$) is either $b-2$ or at most $b-j$. Note that this simplification does not decrease the number of free variables but it reduces the total number of cases.
\begin{lemma}[Extension of Lemma 6 in \cite{costaDalai}]
\label{lemma4}
Suppose that $q_1 \leq q_2 \leq \ldots \leq q_{j-1}$. If all the $p_i$ are less than or equal to $1-\alpha$ where $0\leq\alpha<1$, then
\begin{align}
\Psi_j(p_1, p_2, \ldots, p_{j-1}, 0, \ldots, 0&;q_1, q_2, \ldots, q_b) \nonumber \\
& \leq \Psi_j(1-\alpha, \alpha, 0, \ldots, 0; q_1, q_2, \ldots, q_b).
\label{lemma4ineq1}
\end{align}
\end{lemma}

The following lemma, whose proof can be found in the appendix, takes care of the cases when there is at least one element greater or equal to $1-\epsilon$ in $p$ or $q$ vector. If this element is $p_1$, because of Lemma \ref{lemma3} we can assume $q_1$ is the minimum among the $q$-values if we are maximizing $\Psi_j$. For the evaluation of $\widecheck{M}_1$, this implies that $q_1=0$ whenever $p_1=1-\epsilon$ and vice-versa.
\begin{lemma}
\label{lemma5}
Assume that $\epsilon \leq \frac{1}{j+1}$, $p_1 \geq 1-\epsilon$ and $q_1 \leq q_2 \leq \ldots \leq q_b$. Then
\begin{equation}
\label{lemma5ineq}
\Psi_j(p_1, p_2, \ldots, p_b; q_1, q_2, \ldots, q_b) \leq \Psi_j(p_1, p_2, \ldots, p_b; 0, q_1 + q_2, q_3, \ldots, q_b).
\end{equation}
\end{lemma}

Thanks to Lemmas \ref{lemma4} and \ref{lemma5},  we obtain the following proposition.

\begin{prop}\label{reductions2}
$\widecheck{M}_1$ is attained in one of the following points:
\begin{itemize}
\item[1)]  for $(p;q)$ of the form
$$(\overbrace{0,\dots,0}^{l_1},\underbrace{\alpha,\dots,\alpha}_{l_2},\overbrace{\beta,\dots,\beta}^{b-l_1-l_2-2},0,1-\epsilon;\overbrace{\delta,\dots,\delta}^{l_1},\underbrace{0,\dots,0}_{l_2},\overbrace{\eta,\dots,\eta}^{b-l_1-l_2-2},1-\epsilon,0)$$
where $\alpha,\beta, \delta, \eta \geq 0 $ and $$l_2\alpha+(b-l_1-l_2-2)\beta+(1-\epsilon)=1=l_1\delta+(b-l_1-l_2-2)\eta+(1-\epsilon);$$
\item[2)]  for $(p;q)$ of the form
$$(\overbrace{0,\dots,0}^{l_1},\underbrace{\alpha,\dots,\alpha}_{l_2},\overbrace{\beta,\dots,\beta}^{b-l_1-l_2-1},0;\overbrace{\delta,\dots,\delta}^{l_1},\underbrace{0,\dots,0}_{l_2},\overbrace{\eta,\dots,\eta}^{b-l_1-l_2-1},1-\epsilon)$$
where $\alpha,\beta, \delta, \eta \geq 0 $ and $$l_2\alpha+(b-l_1-l_2-1)\beta=1=l_1\delta+(b-l_1-l_2-1)\eta+(1-\epsilon);$$
\item[3)]  for $(p;q)$ of the form
$$(\overbrace{0,\dots,0}^{l_1},\underbrace{\alpha,\dots,\alpha}_{l_2},\overbrace{\beta,\dots,\beta}^{b-l_1-l_2};\overbrace{\delta,\dots,\delta}^{l_1},\underbrace{0,\dots,0}_{l_2},\overbrace{\eta,\dots,\eta}^{b-l_1-l_2})$$
where $\alpha,\beta, \delta, \eta \geq 0 $ and $$l_2\alpha+(b-l_1-l_2)\beta=1=l_1\delta+(b-l_1-l_2)\eta.$$
\end{itemize}
Moreover, we can assume that the number of zeros that appear in $p$ and in $q$ is either $b-2$ or at most $b-j$.
\end{prop}
\begin{proof}
We consider the finite list of cases provided by Proposition \ref{reductions} and we relax the domains of $p$ and $q$ allowing $\alpha$ and $\delta$ to be 0. Here, due to Lemma \ref{lemma5}, a maximum with exactly one element equal to $1-\epsilon$ in $p$ (resp. $q$) implies a zero in the same coordinate of $q$ (resp. $p$).
Finally, because of Lemma \ref{lemma4}, a maximum with $b-j+1$ or more coordinates in $p$ (resp. $q$) equal to zero is also attained in a point of the form $p = (1-\epsilon, \epsilon, 0, \ldots, 0)$ ($q = (1-\epsilon, \epsilon, 0, \ldots, 0)$).
\end{proof}

\begin{prop}\label{prop:M2u}
$\widecheck{M}_2$ is upper bounded by the global maximum of $\Psi_j$ which is attained in a point  $(p;q)$ of the following form:
$$(\overbrace{0,\dots,0}^{l_1},\underbrace{\alpha,\dots,\alpha}_{l_2},\overbrace{\beta,\dots,\beta}^{b-l_1-l_2};\overbrace{\delta,\dots,\delta}^{l_1},\underbrace{0,\dots,0}_{l_2},\overbrace{\eta,\dots,\eta}^{b-l_1-l_2})$$
where $\alpha,\beta, \delta, \eta \geq 0 $ and $$l_2\alpha+(b-l_1-l_2)\beta=1=l_1\delta+(b-l_1-l_2)\eta.$$
Moreover, we can assume that the number of zeros that appear in $p$ and in $q$ is either $b-1$ or at most $b-j$.
\begin{proof}
In order to find the global maximum of $\Psi_j$ we need no restriction on the pairs $(p,q)$, i.e., $p \in [0,1]^b$ and $q \in [0,1]^b$.
Using Lemmas \ref{lemma1}, \ref{lemma2}, \ref{lemma3} and \ref{lemma4}, we can easily derived the desired points.
\end{proof}
\end{prop}

Now, to provide a list of possible maxima also for the other $\widecheck{M}_i$ and $\widehat{M}_i$, we need also the following additional lemma.

\begin{lemma}
\label{lemma6}
Assume that $\epsilon < \frac{1}{2}$, $q_1 \geq 1-\epsilon$ and $0 < \delta \leq \epsilon$, then
\begin{equation}
\label{lemma6ineq}
\Psi_j(1-\epsilon + \delta, p_2, p_3, \ldots, p_b; q_1, q_2, \ldots, q_b) < \Psi_j(1-\epsilon, p_2+\delta, p_3, \ldots, p_b; q_1, q_2, \ldots, q_b).
\end{equation}
\end{lemma}

Thanks to Lemma \ref{lemma6},  whose proof can be found in the appendix, we obtain the following proposition.

\begin{prop}\label{prop:M3u}
$\widecheck{M}_3$ is attained in a point $(p;q)$ of the following form:
$$(1-\epsilon, \overbrace{0,\dots,0}^{l_1},\underbrace{\alpha,\dots,\alpha}_{l_2},\overbrace{\beta,\dots,\beta}^{b-l_1-l_2-1};1-\epsilon, \overbrace{\delta,\dots,\delta}^{l_1},\underbrace{0,\dots,0}_{l_2},\overbrace{\eta,\dots,\eta}^{b-l_1-l_2-1})$$
where $\alpha,\beta, \delta, \eta \geq 0 $ and $$l_2\alpha+(b-l_1-l_2-1)\beta=\epsilon=l_1\delta+(b-l_1-l_2-1)\eta.$$
Moreover, we can assume that the number of zeros that appear in $p$ and in $q$ is either $b-2$ or at most $b-j$.
\end{prop}
\begin{proof}
In order to find the values $\widecheck{M}_3$ we need to restrict the function $\Psi_j$ to the pairs $(p,q)$ such that $p$ and $q$ belong to $\widecheck{\mathcal{P}}_b^{1}$ (by symmetry we can fix an arbitrary coordinate). 

Using Lemmas \ref{lemma1}, \ref{lemma2}, \ref{lemma3} and \ref{lemma4}, we see that $\widecheck{M}_3$ is attained in a point of the following form:
$$(\gamma, \overbrace{0,\dots,0}^{l_1},\underbrace{\alpha,\dots,\alpha}_{l_2},\overbrace{\beta,\dots,\beta}^{b-l_1-l_2-1};\zeta, \overbrace{\delta,\dots,\delta}^{l_1},\underbrace{0,\dots,0}_{l_2},\overbrace{\eta,\dots,\eta}^{b-l_1-l_2-1})$$
where $\alpha,\beta, \delta, \eta \geq 0$,  $\ \ \gamma, \zeta \geq 1-\epsilon$ and $$\gamma + l_2\alpha+(b-l_1-l_2-1)\beta=1=\zeta + l_1\delta+(b-l_1-l_2-1)\eta.$$
Finally, because of Lemma \ref{lemma6} a maximum with $\gamma, \zeta \geq 1-\epsilon$ is also attained in a point with $\gamma = \zeta = 1-\epsilon$.
\end{proof}
\begin{prop}\label{prop:M4u}
$\widecheck{M}_4$ is attained in one of the following points:
\begin{enumerate}
\item[1)] for $(p;q)$ of the form
$$
(\gamma,\alpha, \ldots, \alpha,0;0,\delta,  \ldots, \delta, \zeta)
$$
where $\alpha,\delta\geq 0 $,  $\gamma, \zeta \geq 1-\epsilon$, and $$\gamma + (b-2) \alpha
=1=(b-2) \delta +\zeta.$$
\item[2)] for $(p;q)$ of the form
$$
(\gamma,\overbrace{0,\dots,0}^{l_1},\underbrace{\alpha,\dots,\alpha}_{b-l_1-2},0;0,\overbrace{\delta,\dots,\delta}^{l_1},\underbrace{\eta,\dots,\eta}_{b-l_1-2},\zeta)
$$
where $\alpha, \delta, \eta \geq 0 $,  $\gamma \geq 1-\epsilon$, $\zeta \in \{1-\epsilon, 1\}$, and $$(b-l_1-2) \alpha+\gamma=1=l_1\delta+(b-l_1-2)\eta+\zeta.$$
\item[3)] for $(p;q)$ of the form
$$
(\gamma,\overbrace{0,\dots,0}^{l_1},\underbrace{\alpha,\dots,\alpha}_{l_2},\overbrace{\beta,\dots,\beta}^{b-l_1-l_2-2},0;0,\overbrace{\delta,\dots,\delta}^{l_1},\underbrace{0,\dots,0}_{l_2},\overbrace{\eta,\dots,\eta}^{b-l_1-l_2-2},\zeta)
$$
where $\alpha,\beta, \delta, \eta \geq 0 $,  $\gamma, \zeta \in \{1-\epsilon, 1\}$, and $$l_2\alpha+(b-l_1-l_2-2)\beta+\gamma=1=l_1\delta+(b-l_1-l_2-2)\eta+\zeta.$$
\end{enumerate}
Moreover, we can assume that the number of zeros that appear in $p$ and in $q$ is either $b-1$ or at most $b-j$.
\end{prop}
\begin{proof}
In order to find the values $\widecheck{M}_4$ we need to restrict the function $\Psi_j$ to the pairs $(p;q)$ such that $p$ belongs to $\widecheck{\mathcal{P}}_b^{1}$ and $q$ belongs to $\widecheck{\mathcal{P}}_b^{b}$ (by symmetry we can choose, arbitrarily, two different coordinates). 

Using Lemmas \ref{lemma1}, \ref{lemma2}, \ref{lemma3}, \ref{lemma4} and \ref{lemma5} we see that $\widecheck{M}_4$ is attained in a point of the following form:
$$
(\gamma,\overbrace{0,\dots,0}^{l_1},\underbrace{\alpha,\dots,\alpha}_{l_2},\overbrace{\beta,\dots,\beta}^{b-l_1-l_2-2},0;0,\overbrace{\delta,\dots,\delta}^{l_1},\underbrace{0,\dots,0}_{l_2},\overbrace{\eta,\dots,\eta}^{b-l_1-l_2-2},\zeta)
$$
where $\alpha,\beta, \delta, \eta \geq 0 $,  $\gamma, \zeta \geq 1-\epsilon$, and $$l_2\alpha+(b-l_1-l_2-2)\beta+\gamma=1=l_1\delta+(b-l_1-l_2-2)\eta+\zeta.$$
Finally, we can split this case into three cases. The first one is for $l_1 = l_2 = 0$, the second one for $l_1>0, l_2 = 0$ and the third one for $l_1, l_2 > 0$.  By symmetry the case $l_1 = 0, l_2 > 0$ is included in the second case.  For the second case, by Lemma \ref{lemma2} it is easy to see that $\delta$ or $\zeta$ must be on the boundary in order to be a valid point for $\widecheck{M}_4$. The same argument can be carried out for the third case which implies that $\gamma, \zeta \in \{1-\epsilon, 1\}$.
\end{proof}
\begin{prop}\label{prop:M1l}
$\widehat{M}_1$ is attained in a point $(p;q)$ of the following form:
$$(\overbrace{\epsilon,\dots,\epsilon}^{l_1},\underbrace{\alpha,\dots,\alpha}_{l_2},\overbrace{\beta,\dots,\beta}^{b-l_1-l_2};\overbrace{\delta,\dots,\delta}^{l_1},\underbrace{\epsilon,\dots,\epsilon}_{l_2},\overbrace{\eta,\dots,\eta}^{b-l_1-l_2})$$
where $\alpha,\beta, \delta, \eta \geq \epsilon$ and $$l_1 \epsilon + \alpha+(b-l_1-l_2)\beta=1=l_2 \epsilon + l_1\delta+(b-l_1-l_2)\eta.$$
\end{prop}
\begin{proof}
In order to find the values $\widehat{M}_1$ we need to restrict the function $\Psi_j$ to the pairs $(p,q)$ such that $p$ and $q$ belong to $\widehat{\mathcal{P}}_b^{0}$.  Using Lemmas \ref{lemma1},  \ref{lemma2} and \ref{lemma3} we obtain the thesis.
\end{proof}
\begin{prop}\label{prop:M2l}
$\widehat{M}_2$ is attained in one of the following points:
\begin{enumerate}
\item[1)]  for $(p;q)$ of the form
$$(\overbrace{\alpha,\dots, \alpha}^{l_1},\underbrace{\beta,\dots,\beta}_{b-l_1};\overbrace{\eta,\dots,\eta}^{l_1},\underbrace{\epsilon,\dots,\epsilon}_{b-l_1})$$
where $0 \leq \alpha \leq \epsilon$, $\ \ \beta \geq 0$, $\ \ \eta \geq \epsilon$, and $$l_1 \alpha+(b-l_1)\beta=1=l_1\eta+(b-l_1) \epsilon.$$
\item[2)]  for $(p;q)$ of the form
 $$(\epsilon,  \overbrace{\alpha,\dots, \alpha}^{l_1},\underbrace{\beta,\dots,\beta}_{b-l_1-1};\zeta,\overbrace{\eta,\dots,\eta}^{l_1},\underbrace{\epsilon,\dots,\epsilon}_{b-l_1-1})$$
where $\alpha,\beta \geq 0$, $\ \ \zeta, \eta \geq \epsilon$, and $$\epsilon + l_1\alpha+(b-l_1-1)\beta=1=\zeta + l_1\eta+(b-l_1-1) \epsilon.$$
\item[3)]  for $(p;q)$ of the form
 $$(\overbrace{0,\dots,0}^{l_1},\underbrace{\alpha,\dots,\alpha}_{l_2},\overbrace{\beta,\dots,\beta}^{b-l_1-l_2};\overbrace{\delta,\dots,\delta}^{l_1},\underbrace{\eta,\dots,\eta}_{l_2},\overbrace{\epsilon,\dots,\epsilon}^{b-l_1-l_2})$$
where $\alpha,\beta \geq 0$, $\ \ \delta, \eta \geq \epsilon $ and $$l_2\alpha+(b-l_1-l_2)\beta=1=l_1\delta+(b-l_1-l_2)\eta.$$
Moreover, we can assume that the number of zeros that appear in $p$ is either $b-1$ or at most $b-j$.
\end{enumerate}
\end{prop}
\begin{proof}
In order to find the values $\widehat{M}_2$ we need to restrict the function $\Psi_j$ to the pairs $(p,q)$ such that $p$ belongs to $\widehat{\mathcal{P}}_b^{1}$ and $q$ belongs to $\widehat{\mathcal{P}}_b^{0}$.  In addition, we relax the domain of $p$ by removing the constraint on $p_1$ to be a minimum coordinate., i.e., $p \in [0, \epsilon] \times
[0,1]^{b-1}$. 
However,  this implies that $p$ belongs to $\widehat{\mathcal{P}}_b^{i}$ for some $i \in [1, b]$.  Therefore, by symmetry, we are still considering valid candidates for $\widehat{M}_2$ under this domain.

Using Lemmas \ref{lemma1}, \ref{lemma2}, \ref{lemma3} and \ref{lemma4}, we see that $\widehat{M}_2$ is attained in a point of the following form:
$$(\gamma,  \overbrace{0,\dots,0}^{l_1}, \underbrace{\alpha,\dots, \alpha}_{b-l_1-l_2-1},\overbrace{\beta,\dots,\beta}^{l_2};\zeta,\overbrace{\delta,\dots,\delta}^{l_1}, \underbrace{\eta,\dots,\eta}_{b-l_1-l_2-1},\overbrace{\epsilon,\dots,\epsilon}^{l_2})$$
where $\alpha,\beta \geq 0$,  $\ \ 0 \leq \gamma \leq \epsilon$,  $\ \ \zeta, \delta, \eta \geq \epsilon$, and $$\gamma + (b-l_1-l_2-1)\alpha+l_2 \beta=1=\zeta + l_1\delta+ (b-l_1-l_2-1)\eta + l_2 \epsilon.$$
Finally,  we can split this case into three cases. The first one is when $l_1 = 0$ and the average between $\gamma$ and the $\alpha$-components is less than or equal to $\epsilon$,  the second one for $\gamma = \epsilon$ and $l_1=0$ while the third one for $\gamma = 0$ and $l_1 \geq 0$.  We have not considered the case $\gamma = \epsilon$ and $l_1 >  0$ since it is a subcase of the third one.
\end{proof}
\begin{prop}\label{prop:M3l}
An upper bound on $\widehat{M}_3$ is obtained by computing the maximum of $\Psi_j$ over points of the following form:
\begin{enumerate}
\item[1)] for $(p;q)$ of the form
$$(\beta, \overbrace{0,\dots,0}^{l_1},\underbrace{\alpha,\dots,\alpha}_{l_2},\overbrace{\beta,\dots,\beta}^{b-l_1-l_2-1};\eta, \overbrace{\delta,\dots,\delta}^{l_1},\underbrace{0,\dots,0}_{l_2},\overbrace{\eta,\dots,\eta}^{b-l_1-l_2-1})$$
where $\alpha, \delta \geq 0$,  $\ \ 0 \leq  \beta, \eta \leq \epsilon$ and $$l_2\alpha+(b-l_1-l_2)\beta=1= l_1\delta+(b-l_1-l_2)\eta.$$
\item[2)]  for $(p;q)$ of the form
$$(\gamma, \overbrace{0,\dots,0}^{l_1},\underbrace{\alpha,\dots,\alpha}_{l_2},\overbrace{\beta,\dots,\beta}^{b-l_1-l_2-1};0, \overbrace{\delta,\dots,\delta}^{l_1},\underbrace{0,\dots,0}_{l_2},\overbrace{\eta,\dots,\eta}^{b-l_1-l_2-1})$$
where $\alpha, \beta, \delta, \eta \geq 0$,  $\ \  0 \leq \gamma \leq \epsilon$ and $$\gamma + l_2\alpha+(b-l_1-l_2-1)\beta=1= l_1\delta+(b-l_1-l_2-1)\eta.$$
\item[3)]  for $(p;q)$ of the form
$$(\gamma, \overbrace{0,\dots,0}^{l_1},\underbrace{\alpha,\dots,\alpha}_{l_2},\overbrace{\beta,\dots,\beta}^{b-l_1-l_2-1};\epsilon, \overbrace{\delta,\dots,\delta}^{l_1},\underbrace{0,\dots,0}_{l_2},\overbrace{\eta,\dots,\eta}^{b-l_1-l_2-1})$$
where $\alpha, \beta, \delta, \eta \geq 0$,  $\ \ 0 \leq  \gamma \leq \epsilon$ and $$\gamma + l_2\alpha+(b-l_1-l_2-1)\beta=1= \epsilon+ l_1\delta+(b-l_1-l_2-1)\eta.$$
\end{enumerate}
Moreover, we can assume that the number of zeros that appear in $p$ and in $q$ is either $b-1$ or at most $b-j$.
\end{prop}
\begin{proof}
In order to find and upper on the values $\widehat{M}_3$ we need to restrict the function $\Psi_j$ to the pairs $(p,q)$ such that $p$ and $q$ belong to $\widehat{\mathcal{P}}_b^{1}$ (by symmetry we can fix an arbitrary coordinate). In addition, we relax the domains of $p$ and $q$ by removing the constraints on $p_1$ and $q_1$ to be minimum components, i.e.,  $p, q \in [0, \epsilon] \times [0,1]^{b-1}$.

Using Lemmas \ref{lemma1}, \ref{lemma2}, \ref{lemma3} and \ref{lemma4},  we see that under this extended domain an upper bound on $\widehat{M}_3$ is attained in a point of the following form:
$$(\gamma, \overbrace{0,\dots,0}^{l_1},\underbrace{\alpha,\dots,\alpha}_{l_2},\overbrace{\beta,\dots,\beta}^{b-l_1-l_2-1};\zeta, \overbrace{\delta,\dots,\delta}^{l_1},\underbrace{0,\dots,0}_{l_2},\overbrace{\eta,\dots,\eta}^{b-l_1-l_2-1})$$
where $\alpha, \beta, \delta, \eta \geq 0$,  $\ \ 0 \leq \gamma, \zeta \leq \epsilon$ and $$\gamma + l_2\alpha+(b-l_1-l_2-1)\beta=1=\zeta+ l_1\delta+(b-l_1-l_2-1)\eta.$$
Finally,  we can split this case into three cases. The first one is when the averages between $\gamma$ and the $\beta$-components and between $\zeta$ and the $\eta$-components are less than or equal to $\epsilon$, the second one for $0 \leq \gamma \leq \epsilon$ and $\zeta = 0$, and the third one for $0 \leq \gamma \leq \epsilon$ and $\zeta = \epsilon$.  By symmetry,  the cases in which $\gamma = 0$ and $0 \leq \zeta \leq \epsilon$ or $\gamma = \epsilon$ and $0 \leq \zeta \leq \epsilon$ are included in the second and third cases.
\end{proof}
\begin{prop}\label{prop:M4l}
$\widehat{M}_4$ is upper bounded by the global maximum of $\Psi_j$ which is attained in a point $(p;q)$ of the following form:
$$(\overbrace{0,\dots,0}^{l_1},\underbrace{\alpha,\dots,\alpha}_{l_2},\overbrace{\beta,\dots,\beta}^{b-l_1-l_2};\overbrace{\delta,\dots,\delta}^{l_1},\underbrace{0,\dots,0}_{l_2},\overbrace{\eta,\dots,\eta}^{b-l_1-l_2})$$
where $\alpha,\beta, \delta, \eta \geq 0 $ and $$l_2\alpha+(b-l_1-l_2)\beta=1=l_1\delta+(b-l_1-l_2)\eta.$$
Moreover, we can assume that the number of zeros that appear in $p$ and in $q$ is either $b-1$ or at most $b-j$.
\begin{proof}
Analogous to the proof of Proposition \ref{prop:M2u}.
\end{proof}
\end{prop}

\begin{rem}

Each configuration that appears in the list of possible maxima in the previous propositions leads to an optimization problem that depends on at most $3$ free variables.  Therefore, for given $b$ and $k$, we can analytically determine, using Mathematica,  those maxima. 
\label{rem:3vars}
\end{rem}

The previous propositions allow us to determine a finite list of maxima candidates for each $\widecheck{M}_i$ and $\widehat{M}_i$. We have analytically determined and inspected using Mathematica all the possible maximum points. 
We have restricted our attention to $(b,k)$-cases for small $b$ and $k$ in order to avoid excessive computational complexity. It is important to note that for the $(b,k)$-cases that we have considered (see Propositions \ref{Mabove} and \ref{Mbelow}) the global maximum of $\Psi_j$,  for $j=k-2$, satisfy the domains of $\widecheck{M}_2$  and $\widehat{M}_4$. Therefore for these particular cases, we are not upper bounding the values of $\widecheck{M}_2$  and $\widehat{M}_4$ but, instead, we are computing the exact values.  Based on the results of computations, we choose the values of $j$ and $\epsilon$ for each $(b,k)$-case to improve the current best-known bounds on $R_{(b,k)}$.  A more careful choice of these parameters could lead to better bounds except for the case $b=k=6$ (see Remark \ref{rem3}).

\begin{prop}\label{Mabove}
For $j=k-2$, for the values of $\epsilon$ shown, the $\widecheck{M}_i$'s are as shown in the following table


\begin{table}[ht!]
\def\arraystretch{1.}
\centering
\begin{tabularx}{\linewidth}{c@{\extracolsep{\fill}}c@{\extracolsep{\fill}}c@{\extracolsep{\fill}}c@{\extracolsep{\fill}}c@{\extracolsep{\fill}}c@{\extracolsep{\fill}}l}
$(b,k)$ & $\epsilon$ & $\widecheck{M}_1$ & $\widecheck{M}_2$ & $\widecheck{M}_3$ & $\widecheck{M}_4$\\
\hline
$(7,7)$ & $9/100$ & 0.085679 & 0.092593 & 0.000006 & 0.000107\\
$(8,8)$ & $3/25$ & 0.038453 & 0.042840 & 0.000002 & 0.000022\\
$(9,8)$ & $1/10$ & 0.075870 & 0.076905 & 0.000001 & 0.000015\\
$(10,9)$ & $1/15$ & 0.036289 & 0.037935 & $3.4 \cdot 10^{-9}$ & $8.5 \cdot 10^{-8}$\\
$(11,10)$ & $1/11$ & 0.016928 & 0.018144 & $1.4 \cdot 10^{-9}$ & $2.7 \cdot 10^{-8}$\\
$(12,10)$ & $1/20$ & 0.030945 & 0.031036 & $2.1 \cdot 10^{-11}$ & $7.0 \cdot 10^{-9}$\\
$(13,11)$ & $1/25$ & 0.015057 & 0.015473 & $7.8 \cdot 10^{-14}$ & $3.5 \cdot 10^{-12}$\\
$(14,12)$ & $1/13$ & 0.007176 & 0.007529 & $1.2 \cdot 10^{-12}$ & $2.6 \cdot 10^{-11}$\\
$(15,13)$ & $1/12$ & 0.003360 & 0.003588 & $1.1 \cdot 10^{-13}$ & $2.3 \cdot 10^{-12}$\\
\hline
\end{tabularx}
\def\arraystretch{1.35}
\begin{tabularx}{\linewidth}{l}
$\widecheck{M}_1$ is attained at  $(\frac{1}{b}, \ldots, \frac{1}{b}; \frac{1}{b}, \ldots, \frac{1}{b})$\\
$\widecheck{M}_2$ is attained at $(1, 0, \ldots, 0; 0, \frac{1}{b-1}, \ldots, \frac{1}{b-1})$\\
$\widecheck{M}_3$ is attained at $(1-\epsilon, \frac{\epsilon}{b-1}, \ldots, \frac{\epsilon}{b-1}; 1-\epsilon, \frac{\epsilon}{b-1}, \ldots, \frac{\epsilon}{b-1})$\\
$\widecheck{M}_4$ is attained at $(1-\epsilon, \frac{\epsilon}{b-2}, \ldots, \frac{\epsilon}{b-2}, 0; 0, \frac{\epsilon}{b-2}, \ldots, \frac{\epsilon}{b-2},  1-\epsilon)$\\
\hline
\end{tabularx}
\end{table}
\end{prop}
\newpage
\begin{prop}\label{Mbelow}
For $j=3$, $(b,k)=(5,5)$ and $\epsilon = \frac{1}{44}(4+\sqrt{5})$,  the $\widehat{M}_i$'s are as shown in the following table
%
\begin{table}[ht!]
\def\arraystretch{1.5}
\begin{tabularx}{\linewidth}{cl@{\extracolsep{\fill}}l}
$\widehat{M}_i$ & Attained at point $(p;q)$ & Value $\approx$ \\
\hline
$\widehat{M}_1$ & $(\epsilon,\frac{1-\epsilon}{b-1}, \ldots, \frac{1-\epsilon}{b-1}; \gamma, \delta, \ldots, \delta), \delta \approx 0.185275$ & 0.384033\\

$\widehat{M}_2$ & $(0, \frac{1}{b-1}, \ldots, \frac{1}{b-1}; \gamma, \delta, \ldots, \delta), \delta = \epsilon$ & 0.389226\\

$\widehat{M}_3$ & $(\epsilon, \frac{1-\epsilon}{b-2}, \ldots, \frac{1-\epsilon}{b-2}, 0; \epsilon, \alpha, \ldots, \alpha, \beta), \beta \approx 0.454183$ & 0.374759\\

$\widehat{M}_4$ & $(0, \frac{1}{b-1}, \ldots, \frac{1}{b-1}; \gamma, \delta, \ldots, \delta), \delta = \epsilon$ & 0.389226\\
\hline
\end{tabularx}
\end{table}
\\
For $j=3$, $(b,k)=(6,5)$ and $\epsilon = \frac{1}{10}$,  the $\widehat{M}_i$'s are as shown in the following table
\begin{table}[ht!]
\def\arraystretch{1.5}
\begin{tabularx}{\linewidth}{cl@{\extracolsep{\fill}}l}
$\widehat{M}_i$ & Attained at point $(p;q)$ & Value $\approx$ \\
\hline
$\widehat{M}_1$& $(\epsilon,\frac{1-\epsilon}{b-1}, \ldots, \frac{1-\epsilon}{b-1}; \gamma, \delta, \ldots, \delta), \delta \approx 0.153159$ & 0.555625\\

$\widehat{M}_2$ & $(0, \frac{1}{b-1}, \ldots, \frac{1}{b-1}; \gamma, \delta, \ldots, \delta), \delta \approx 0.130217$ & 0.558467\\

$\widehat{M}_3$ & $(\epsilon, \frac{1-\epsilon}{b-2}, \ldots, \frac{1-\epsilon}{b-2}, 0; \epsilon, \alpha, \ldots, \alpha, \beta), \beta \approx 0.376930$  & 0.535106\\

$\widehat{M}_4$ & $(0, \frac{1}{b-1}, \ldots, \frac{1}{b-1}; \gamma, \delta, \ldots, \delta), \delta \approx 0.130217$ & 0.558467\\
\hline
\end{tabularx}
\end{table}
\\
For $j=4$, $(b,k)=(6,6)$ and $\epsilon = \frac{1}{20}$,  the $\widehat{M}_i$'s are as shown in the following table
\begin{table}[ht!]
\def\arraystretch{1.5}
\begin{tabular*}{\linewidth}{cl@{\extracolsep{\fill}}l}
$\widehat{M}_i$ & Attained at point $(p;q)$ & Value $\approx$ \\
\hline
$\widehat{M}_1$& $(\frac{1}{b}, \ldots, \frac{1}{b}; \frac{1}{b}, \ldots, \frac{1}{b})$ & 0.185185\\

$\widehat{M}_2$ & $(\epsilon, \frac{1-\epsilon}{b-1}, \ldots, \frac{1-\epsilon}{b-1}; \gamma, \delta, \ldots, \delta), \delta \approx 0.147757$ & 0.178857\\

$\widehat{M}_3$ & $(\epsilon, 0, \frac{1-\epsilon}{b-2}, \ldots, \frac{1-\epsilon}{b-2}; 0, 1, 0, \ldots, 0)$  & 0.140664\\

$\widehat{M}_4$ & $(1, 0, \ldots, 0; 0, \frac{1}{b-1}, \ldots, \frac{1}{b-1})$ & $0.192000$\\
\hline
\end{tabular*}
\end{table}
\\
\end{prop}
\begin{rem}
The values reported for $\widehat{M}_3$ are not approximate values of the exact values of $\widehat{M}_3$ but upper bounds. 
We point out that the value $\widehat{M}_1$ for $b=k=6$ is only attained for uniform distributions. This will be important for a qualitative analysis of our bound on $R_{(b,k)}$ for different values of $b$ and $k$, see Section \ref{sec:qualk6}.
\label{rem:uniquemax}
\end{rem}

As a consequence of Propositions \ref{Mabove}, \ref{Mbelow} and equation \eqref{eq:QuadracBoundedbyMs} we are able to evaluate the values of $M$ for both the partitions $\{\widecheck{P}_b^i\}_{i=0,\ldots,b}$ and $\{\widehat{P}_b^i\}_{i=0,\ldots,b}$. Then we state the following theorem
\begin{teo}\label{theorem:Msvalues}
	Using the partition $\{\widecheck{P}_b^i\}_{i=0,\ldots,b}$ we have
	
	\begin{table}[ht!]
		\def\arraystretch{1.5}
		\begin{tabular*}{\linewidth}{c@{\extracolsep{\fill}}c@{\extracolsep{\fill}}c@{\extracolsep{\fill}}c@{\extracolsep{\fill}}c@{\extracolsep{\fill}}c@{\extracolsep{\fill}}c@{\extracolsep{\fill}}c@{\extracolsep{\fill}}c@{\extracolsep{\fill}}c}
		$(b,k)$ & $(7,7)$ & $(8,8)$ & $(9,8)$ & $(10,9)$ & $(11,10)$ \\
		\hline
		$M$ & $\approx 0.0861594$ & $\approx 0.0388599$ & $\approx 0.0758830$ & $\approx 0.0363565$ & $\approx  0.0170049$\\
		\hline \\
		$(b,k)$ & $(12,10)$ & $(13,11)$ & $(14,12)$ & $(15,13)$ & \\
		\hline
		$M$ & $\approx 0.0309448$ & $\approx 0.0150674$ & $\approx 0.0071917$ & $\approx 0.0033733$\\
		\hline
		\end{tabular*}
	\end{table}

	
	Using the partition $\{\widehat{P}_b^i\}_{i=0,\ldots,b}$ we have

	\begin{table}[ht!]
		\def\arraystretch{1.5}
		\begin{tabular*}{\linewidth}{c@{\extracolsep{\fill}}c@{\extracolsep{\fill}}c@{\extracolsep{\fill}}c}
		$(b,k)$ & $(5,5)$ & $(6,5)$ & $(6,6)$\\
		\hline
		$M$ & $\approx 0.3873676$ & $\approx 0.5567010$ & $\frac{5}{27} \approx 0.185185$\\
		\hline
		\end{tabular*}
	\end{table}		
	
\end{teo}


\begin{rem}\label{rem3}
For the underlined $(b,k)$-cases reported in Table \ref{tab:bkbounds}, it is interesting to note that the maximum in \eqref{eq:PsiMax} is only achieved for uniform distributions. This means that, for these particular cases, any new upper bounds that can be found on the quadratic form in equation \eqref{simmetrizzata2} cannot further improve those bounds. Note that, for $(b,4)$-cases when $b \geq 4$, the maximum of the quadratic form in \eqref{simmetrizzata2} is only achieved for uniform distributions if we suppose that the frequency of each symbol in all the coordinates of the code is less than or equal to $1/2$. For the special case $b=k=6$, the values we obtained for the $M_i$ constants are such that the resulting $\eta_0$ in the statement of Lemma \ref{lemmaMax} is equal to $1$. That is, the constant $M$ is actually $M_1$, which means that in our bound the worst-case scenario is given by the balanced subcodes. The resulting value $M_1=5/27$ is actually the value attained by $\Psi_j(p,q)$ for uniform $p$ and $q$. Roughly speaking, this should be interpreted as saying that our procedure is unable to give for $R_{(6,6)}$ a bound smaller than $5/59$ because such a rate might in principle be achieved if all subcodes have a uniform distribution on each coordinate. However, for such globally balanced codes, one can use a different argument based on the minimum distance of the code to get even stronger upper bounds on $R_{(b,k)}$. In the next section, we combine the two procedures to deduce a rigorous proof that indeed the bounds shown in Table \ref{tab:bkbounds} are not sharp for different values of $b$ and $k$.
\end{rem}

\section{A qualitative analysis on \texorpdfstring{$R_{(b,k)}$}{Lg}}
\label{sec:qualk6}
In this section we show that, at least for the underlined $(b,k)$-cases in Table \ref{tab:bkbounds} and for case $(b,k)=(6,6)$, the bound in equation \eqref{Rbk_Mj_bound}, for $j=k-2$ with $\mathbf{M}_{k-2}=\Psi_{k-2}(1/b, \ldots, 1/b;$ $1/b, \ldots, 1/b)$, is not sharp. We also show that, the bound given in equation \eqref{eq:dalaivenkatbound} is not sharp for every $(b,4)$-cases with $b \geq 5$ and $j=2$.
In this discussion, we use continuity arguments whose quantitative analysis would require long and complicated computations. For this reason, we do not provide explicit numerical improvements on $R_{(b,k)}$ and only show that the bounds on $R_{(b,k)}$ can be improved.

To prove our statement, we invoke an upper bound from \cite{Aaltonen} on the minimum hamming distance $d_H(C)$ of a $b$-ary code $C$ with a given rate $R$. It suffices here  to mention that, set $\delta:=d_H(C)/n$, this bound is of the form $\delta\leq F(R)$ for a suitable decreasing continuous function $F$.
Due to the monotonicity of $F$ there exists a maximum value of $R$ for which the inequality $R \leq \frac{(b-2)!}{(b-k+1)! b^{k-3}} F(R)$ is satisfied. \\
Using Mathematica on the specific bound in \cite{Aaltonen}, one finds that
\begin{align*}
R\leq \frac{(b-2)^{\underline{k-3}}}{b^{k-3}} F(R) \implies R < U_{(b,k)}
\end{align*}
where $(b-2)^{\underline{k-3}} = (b-2) \cdots (b-k+2)$ and $U_{(b,k)}$ takes the values shown in Table \ref{tab:ubk} for some $(b,k)$ pairs. Note that most of these pairs are actually those underlined in Table \ref{tab:bkbounds}.
\begin{table}[ht!]
\def\arraystretch{1.5}
\caption{$U_{(b,k)}$ values}
\begin{tabularx}{\linewidth}{c@{\extracolsep{\fill}}c@{\extracolsep{\fill}}c@{\extracolsep{\fill}}c@{\extracolsep{\fill}}c}
\hline
$U_{(6,6)} = 0.08469$ & $U_{(7,6)} = 0.13440$ & $U_{(8,6)} = 0.18125$ & $U_{(9,6)} = 0.22405$ \\ $U_{(10,6)} = 0.26268$ & $U_{(11,6)} = 0.29744$ & $U_{(12,6)} = 0.32874$ & $U_{(13, 6)} = 0.35699$ \\ $U_{(14,6)} = 0.38258$ & $U_{(8,7)} = 0.07200$ & $U_{(9,7)} = 0.10510$ &$U_{(10,7)} = 0.13822$ \\ $U_{(11,7)} = 0.17025$ & $U_{(12,7)} = 0.20068$ & $U_{(13,7)} = 0.22930$ & $U_{(14,7)} = 0.25609$ \\ $U_{(10, 8)} = 0.05749$ & $U_{(11, 8)} = 0.08043$ & $U_{(12, 8)} = 0.10419$ & $U_{(13, 8)} = 0.12808$ \\ $U_{(14, 8)} = 0.15163$ &  $U_{(11, 9)} = 0.03006$ & $U_{(12, 9)} = 0.04465$ & $U_{(13, 9)} =0.06081$ \\ $U_{(14, 9)} = 0.07799$ & $U_{(13, 10)} = 0.02386$ & $U_{(14, 10)} = 0.03412$ & $U_{(14, 11)} = 0.01236$ \\
\hline
\end{tabularx}
\label{tab:ubk}
\end{table}

Because of the continuity of $F$, this implies that there exist $\epsilon_1>0$ and $\epsilon_2>0$ such that
\begin{align*}
R \leq \left(\frac{(b-2)^{\underline{k-3}}}{b^{k-3}}+\epsilon_1\right) F(R)+\epsilon_2 \implies R < U_{(b,k)} + 10^{-5}
\end{align*} 

We note that, if $p_1=p_2=\dots=p_b=1/b$, given $i\not=h\in [1,b]$ and chosen at random $x_1, \ldots, x_{k-4}, z \in [1,b]$ according to the distribution $p$, the probability that $i,h,x_1, \ldots, x_{k-4}$, $z$ are all different is $(b-2)^{\underline{k-3}}/b^{k-3}$. 
Therefore, by continuity, there exists $\epsilon_3>0$ such that given $i\not=h\in [1,b]$ and chosen at random $x_1, \ldots, x_{k-4}, z \in [1,b]$ according to the distribution $p'$ where $p'_1,p'_2,\dots,p'_b\in [1/b-\epsilon_3,1/b+\epsilon_3]$, the probability that $i,h,x_1, \ldots, x_{k-4}, z$ are all different is less than $(b-2)^{\underline{k-3}}/b^{k-3}+\epsilon_1$. Now we divide the coordinates $i\in [1,n]$ into two sets according to whether the distribution $f_i$ has all its values in $[1/b-\epsilon_3,1/b+\epsilon_3]$ or not. More precisely, we define:
\begin{align*}U_{\epsilon_3}:=\{i\in [1,n]: f_{i,h} \in [1/b-\epsilon_3,1/b+\epsilon_3], \ \forall h\in[1,b] \}.\end{align*}
We can assume, up to reordering the coordinates, that $U_{\epsilon_3}=[1,t]$ for some value $t$. Here we divide the discussion into two cases, according to the size of $t$, and we show that in both cases a better bound on $R_{(b,k)}$ can be obtained.
\subsubsection*{A) Let us assume that $t<n(1-\epsilon_2)$} As a consequence of Hansel's Lemma, we have the following 
\begin{align*}\log(|C|)&\leq \left(1+o(1)\right)\frac{1}{2}\sum_{i\in [\ell+1,n]}\sum_{\omega,\mu\in \Omega}\lambda_{\omega}\lambda_{\mu}\Psi_{k-2}(f_{i|\omega},f_{i|\mu}) \\ &\leq  \left(1+o(1)\right)\frac{1}{2}\left[ \sum_{i\in [\ell+1,t]}M+\sum_{i\in [t+1,n]}\sum_{\omega,\mu\in \Omega}\lambda_{\omega}\lambda_{\mu}\Psi_{k-2}(f_{i|\omega},f_{i|\mu})\right]\end{align*}
where $M$ is the upperbound of equation \eqref{eq:QuadracBoundedbyMs} given in Theorem \ref{theorem:Msvalues}.
Due to the following lemma, we are able to provide a better upper bound to the second term of the sum. 
\begin{lemma}
Assume that $f_{i,h}\not \in [1/b-\epsilon_3,1/b+\epsilon_3]$ for some $h\in [1,b]$. Then there exists $M'<M$ such that:
$$
\sum_{\omega,\mu\in \Omega}\lambda_{\omega}\lambda_{\mu}\Psi_{k-2}(f_{i|\omega},f_{i|\mu})\leq M'.
$$
\end{lemma}

\begin{proof}
Consider first for simplicity the case when $f_{i,h}<1/b-\epsilon_3$.  Let $\Omega'\subseteq \Omega$ be the subset of the $\omega$ for which $f_{i,h|\omega}\geq 1/b-\epsilon_3/2$. Then, since
\begin{align*}
f_{i,h}  = \sum_{\omega\in \Omega} \lambda_{\omega}f_{i,h|\omega}
 \geq  \sum_{\omega\in \Omega'} \lambda_{\omega}f_{i,h|\omega}
 \geq  ( 1/b-\epsilon_3/2)\sum_{\omega\in\Omega'} \lambda_{\omega}\,,
\end{align*}
we deduce that 
$$
\sum_{\omega\in\Omega'} \lambda_{\omega} \leq \frac{1/b-\epsilon_3}{1/b-\epsilon_3/2}\,.
$$
From Remarks \ref{rem:ratesunderlined}, \ref{rem:uniquemax} and \ref{rem3}, we know that the maximum of the quadratic form in \eqref{simmetrizzata2} is only achieved for uniform distributions. This means that $M=\Psi_{k-2}(1/b, \ldots, 1/b; 1/b, \ldots, 1/b)$ for the $(b,k)$-cases under consideration. Therefore there is some constant $M_{\epsilon_3}<M$  such that if either $f_{i,h|\omega}$ or $f_{i,h|\mu}$ are not $\Omega'$, then $\Psi_{k-2}(f_{i|\omega},f_{i|\mu})\leq M_{\epsilon_3}$.
This implies
\begin{equation*}
\sum_{\omega,\mu\in \Omega}\lambda_{\omega}\lambda_{\mu}\Psi_{k-2}(f_{i|\omega},f_{i|\mu})\leq \left( \frac{1/b-\epsilon_3}{1/b-\epsilon_3/2}\right)^2 M+\left(1- \left( \frac{1/b-\epsilon_3}{1/b-\epsilon_3/2}\right)^2\right) M_{\epsilon_3}
\end{equation*}
and hence the statement of the lemma for the case $f_{i,h}<1/b-\epsilon_3$. A similar proof holds for $f_{i,h}>1/b+\epsilon_3$.
\end{proof}
In the case $\ell\geq t$, we immediately obtain that $\log(|C|)\leq (n-\ell)\frac{M'}{2}$ which leads to the upperbound $R<\frac{M'}{2+M'}$ that is better than the one shown in Table \ref{tab:bkbounds}. So we can assume $\ell<t$ and therefore:
\begin{align*}\log(|C|)\leq (t-\ell)\frac{M}{2}+(n-t)\frac{M'}{2}\leq (n-n\epsilon_2-\ell)\frac{M}{2}+n\epsilon_2\frac{M'}{2}.\end{align*}
Since $\ell=\left\lfloor{\frac{nR-2\log{n}}{\log\left(2+\bar{\epsilon}\right)}}\right\rfloor=\left\lfloor{nR-2\log{n}}\right\rfloor(1+o(1))$, dividing by $n$ we get:
\begin{align*}R\leq \frac{1}{2}\left[\frac{M(1-\epsilon_2-R+2\frac{\log{n}}{n})}{1-R+2\frac{\log{n}}{n}}+\frac{M'\epsilon_2}{1-R+2\frac{\log{n}}{n}}\right](1-R+2\frac{\log{n}}{n})(1+o(1)).\end{align*}
Set $M''=\frac{M(1-\epsilon_2-R)}{1-R}+\frac{M'\epsilon_2}{1-R}$ we have that $M''<M$ and, taking $n\to \infty$, we obtain:
$$R\leq \frac{M''}{2}(1-R)(1+o(1)).$$
It means that $R<\frac{M''}{2+M''}$ and since $M''<M$, it follows that the bound is not sharp under the assumption of the case $A$.
\subsubsection*{B) Let us assume that $t\geq n(1-\epsilon_2)$} Let us fix two words $u,u'\in C$ at minimum hamming distance and let us choose at random $x,y$. Because of Hansel Lemma we have that:
\begin{align*}\log(|C|)\leq\sum_{i=1}^n \mathbb{E}[\tau(G_i^{u,u',x_1, \ldots, x_{k-4}})].\end{align*}
We recall that $0\leq \tau(G_i^{u,u',x_1,\ldots, x_{k-4}})\leq 1$ and, if $u_i\not=u'_i$, $\tau(G_i^{u,u',x_1, \ldots, x_{k-4}})$ is the probability that given $z\not \in \{u,u',x_1, \ldots,x_{k-4}\}$ we have that $u_i,u'_i,x_{1i},x_{(k-4)i},z_i$ are all different. Since we have chosen at random also $x_1, \ldots, x_{k-4}$, $\mathbb{E}[\tau(G_i^{u,u',x_1, \ldots, x_{k-4}})]$ 
coincides with the probability that given $x_1, \ldots, x_{k-4}, z\not \in \{u,u'\}$ we have that $u_i,u'_i,x_{1i},x_{(k-4)i},z_i$ are all different. Therefore $\mathbb{E}[\tau(G_i^{u,u',x_1, \ldots,x_{k-4}})]\leq (b-2)^{\underline{k-3}}/b^{k-3}+\epsilon_1$ for any $i\in [1,t]$ when $u_i\not=u'_i$, otherwise if $u_i=u'_i$ then the expected value is $0$. 
This means that
\begin{align*}
\log(|C|)&\leq \left(\frac{(b-2)^{\underline{k-3}}}{b^{k-3}}+\epsilon_1\right)\sum_{i=1}^t 1_{u_i\not=u'_i}+\sum_{i=t+1}^n 1 \\
&\leq  \left(\frac{(b-2)^{\underline{k-3}}}{b^{k-3}}+\epsilon_1\right)\sum_{i=1}^n 1_{u_i\not=u'_i}+\sum_{i=n(1-\epsilon_2)}^n 1
\end{align*}
and hence
\begin{align*}\log(|C|)\leq \left(\frac{(b-2)^{\underline{k-3}}}{b^{k-3}}+\epsilon_1\right)d_H(u,u')+n\epsilon_2.\end{align*}
Dividing by $n$ and remembering that $u,u'$ are at minimal hamming distance, we obtain that:
\begin{align*}R\leq \left(\frac{(b-2)^{\underline{k-3}}}{b^{k-3}}+\epsilon_1\right)\delta+\epsilon_2\leq \left(\frac{(b-2)^{\underline{k-3}}}{b^{k-3}}+\epsilon_1\right)F(R)+\epsilon_2.\end{align*}
But, because of the definition of $\epsilon_1$ and $\epsilon_2$, this implies that $R < U_{(b,k)} + 10^{-5}$. It can be easily checked that the bound in Theorem \ref{NotSharp} is strictly greater than $U_{(b,k)} + 10^{-5}$ for every $(b,k)$-cases under consideration and therefore:
\begin{teo}\label{NotSharp}
	\begin{align*}
		R_{(b, k)} < \left(\frac{1}{\log\frac{b}{k-3}} + \frac{b^{k-1}}{b^{\underline{k-1}} \log (b-k+2)} \right)^{-1}
	\end{align*}
	for the $(b,k)$-cases shown in Table \ref{tab:ubk}.
\end{teo}

For cases $(b, 4)$ when $b \geq 4$ we know thanks to \cite{DalaiVenkatJaikumar} that the maximum of \eqref{eq:exptaugf}, under the constraint that $f_{i,a} \leq \frac{1}{2}$ for every $i=1,\ldots, n$ and every $a=1, \ldots, b$, is only achieved for uniform distributions. Therefore we can use the Plotkin bound instead of the Aaltonen bound in order to prove that bound \eqref{eq:dalaivenkatbound} is not sharp when $k=4$ and $b \geq 5$. 

Let $C$ be a $(b,4)$-hash code with rate $R$ and suppose that the frequency of the symbols in all the coordinates of $C$ is uniform. Then by Hansel we get
\begin{equation}\label{eq:b4hansel}
	R \leq \frac{b-2}{b} \cdot \delta,
\end{equation}
where $\delta = d_H(C) / n$. The Plotkin bound for $q$-ary codes with $\delta \leq (b-1)/b$ is the following
\begin{equation}\label{eq:b4plotkin}
	R \leq \log b \left(1- \delta \cdot \frac{b}{b-1} \right).
\end{equation}
Since equation \eqref{eq:b4hansel} is increasing in $\delta$ while \eqref{eq:b4plotkin} is decreasing then we can combine the two bounds to get
\begin{equation}\label{eq:plotkin}
	R \leq \frac{b(b-1) \log b}{(b-1)(b-2) + b^2 \log(b)}.
\end{equation}
It is easy to see that the bound given in \eqref{eq:dalaivenkatbound} for $k=4$ is always strictly greater than \eqref{eq:plotkin} for every $b > 4$. Then, by a continuity argument (as done previously) one can show that bound \eqref{eq:dalaivenkatbound} for $k=4$ is not sharp for every $b \geq 5$. Therefore we have the following theorem.
\begin{teo}\label{NotSharp4}
For every integer $b > 4$
$$
	R_{(b,4)}  < \left(\frac{1}{\log b} + \frac{b^2}{(b^2-3b+2) \log (b-2)}  \right)^{-1}.
$$
\end{teo}

\newpage
\appendix
\section*{Appendix}

Here we provide the proofs of Lemmas \ref{lemma4}, \ref{lemma5}, and \ref{lemma6} stated in Section \ref{CaseReduction}.

\begin{proof}[\textbf{Proof of Lemma \ref{lemma4}}]
Let $0 \leq \delta \leq p_2$. We first prove that
\begin{multline}
\label{lemma4ineq2}
\Psi_j(p_1, p_2, \ldots, p_{j-1}, 0, \ldots, 0; q_1, q_2, \ldots, q_b)\\ \leq
\Psi_j(p_1 + \delta, p_2-\delta, \ldots, p_{j-1}, 0, \ldots, 0; q_1, q_2, \ldots , q_b).
\end{multline}
Using the definition of $\Psi_j$ in eq. \eqref{eq:defPsi}, inequality (\ref{lemma4ineq2}) can be restated by only considering the terms in the summation which differ in the two sides, that is, those corresponding to permutations $\sigma$ such that $1 \in \{ \sigma(1), \ldots, \sigma(j) \}$, $\sigma(j+1)=2$ and $2 \in \{ \sigma(1), \ldots, \sigma(j) \}$, $\sigma(j+1)=1$. This gives
\begin{multline*}
(p_1q_2 + p_2q_1) \sum_{\sigma \in Sym(3,\ldots, b)} q_{\sigma(3)} \cdots q_{\sigma(j+1)}
\\ \leq
((p_1+\delta)q_2 + (p_2-\delta)q_1) \sum_{\sigma \in Sym(3,\ldots, b)} q_{\sigma(3)} \cdots q_{\sigma(j+1)}\,.
\end{multline*}
Rearranging the terms we have
\begin{equation*}
\delta (q_2 - q_1) \sum_{\sigma \in Sym(3,\ldots, b)} q_{\sigma(3)} \cdots q_{\sigma(j+1)} \geq 0\,.
\end{equation*}
Therefore, inequality (\ref{lemma4ineq2}) is thus satisfied since $q_1 \leq q_2$ and $\delta \geq 0$. Moreover, given $h>i$ and $\delta$ such that $0 \leq \delta \leq p_h$, with the same argument we have
\begin{multline}
\label{lemma4ineq3}
\Psi_j(p_1,\ldots, p_i, \ldots, p_h,\ldots, p_{j-1}, 0, \ldots, 0; q_1, q_2, \ldots,q_b)\\ \leq
\Psi_j(p_1,\ldots, p_i+\delta, \ldots, p_h-\delta,\ldots, p_{j-1}, 0, \ldots, 0; q_1, q_2, \ldots,q_b).
\end{multline}
Using multiple times inequality (\ref{lemma4ineq3}) we get the following chain of inequalities
\begin{multline*}
\Psi_j(p_1,p_2,\ldots, p_{j-1}, 0, \ldots, 0; q_1, q_2, \ldots,q_b)\\ \leq
\Psi_j(p_1,\alpha,p'_3,\ldots, p'_{j-1}, 0, \ldots, 0; q_1, q_2, \ldots,q_b) \\ \leq
\Psi_j(1-\alpha,\alpha,0, \ldots, 0; q_1, q_2, \ldots,q_b) \,,
\end{multline*}
where $p'_3 + \ldots + p'_{j-1} = 1-p_1-\alpha$ and $p'_i \in [0,  1-\alpha]$ for $i = 3, \ldots, j-1$.
\end{proof}

\begin{proof}[\textbf{Proof of Lemma \ref{lemma5}}]
Using the definition of $\Psi_j$ in eq. \eqref{eq:defPsi}, inequality (\ref{lemma5ineq}) can be restated by only considering the terms in the summation which differ in the two sides, that is, those corresponding to permutations $\sigma$ such that $1 \in \{ \sigma(1), \ldots, \sigma(j) \}$, $\sigma(j+1)=2$ and $2 \in \{ \sigma(1), \ldots, \sigma(j) \}$, $\sigma(j+1)=1$ and $\{1, 2\} \subseteq \{ \sigma(1), \ldots, \sigma(j) \}$.
Hence inequality (\ref{lemma5ineq}) becomes:
\begin{align*}
&(p_1q_2 + p_2q_1)\sum_{\sigma \in Sym(3,\ldots, b)} p_{\sigma(3)} \cdots p_{\sigma(j+1)} + q_{\sigma(3)} \cdots q_{\sigma(j+1)} + \\
&(j-1)q_2q_1 \sum_{\sigma \in Sym(3,\ldots, b)} q_{\sigma(3)} \cdots q_{\sigma(j)} p_{\sigma(j+1)}
\\
&\leq p_1 (q_1 + q_2)\sum_{\sigma \in Sym(3,\ldots, b)} p_{\sigma(3)} \cdots p_{\sigma(j+1)} + q_{\sigma(3)} \cdots q_{\sigma(j+1)}.
\end{align*}
That is
\begin{align*}
&p_2q_1\sum_{\sigma \in Sym(3,\ldots, b)} p_{\sigma(3)} \cdots p_{\sigma(j+1)} + q_{\sigma(3)} \cdots q_{\sigma(j+1)} +
\\
\nonumber
&(j-1)q_2q_1 \sum_{\sigma \in Sym(3,\ldots, b)} q_{\sigma(3)} \cdots q_{\sigma(j)} p_{\sigma(j+1)}
\\
&\leq p_1 q_1\sum_{\sigma \in Sym(3,\ldots, b)} p_{\sigma(3)} \cdots p_{\sigma(j+1)} + q_{\sigma(3)} \cdots q_{\sigma(j+1)}.
\end{align*}
We have
\begin{align*}
&p_2q_1\sum_{\sigma \in Sym(3,\ldots, b)} p_{\sigma(3)} \cdots p_{\sigma(j+1)} + q_{\sigma(3)} \cdots q_{\sigma(j+1)} +
\\
&(j-1)q_2q_1 \sum_{\sigma \in Sym(3,\ldots, b)} q_{\sigma(3)} \cdots q_{\sigma(j)} p_{\sigma(j+1)}
\\
& \overset{(i)}{\leq}
q_1 \epsilon \sum_{\sigma \in Sym(3,\ldots, b)} p_{\sigma(3)} \cdots p_{\sigma(j+1)} + q_{\sigma(3)} \cdots q_{\sigma(j+1)} + \\ &(j-1) q_1 \epsilon \sum_{\sigma \in Sym(3,\ldots, b)} q_2 q_{\sigma(3)} \cdots q_{\sigma(j)} \nonumber
\\
& \overset{(ii)}{\leq}
q_1 \epsilon \sum_{\sigma \in Sym(3,\ldots, b)} p_{\sigma(3)} \cdots p_{\sigma(j+1)} + j q_1 \epsilon \sum_{\sigma \in Sym(3,\ldots, b)} q_{\sigma(3)} \cdots q_{\sigma(j+1)}
\\
& \overset{(iii)}{\leq}
(1-\epsilon) q_1\sum_{\sigma \in Sym(3,\ldots, b)} p_{\sigma(3)} \cdots p_{\sigma(j+1)} + q_{\sigma(3)} \cdots q_{\sigma(j+1)}
\\
& \overset{(iiii)}{\leq}
p_1 q_1\sum_{\sigma \in Sym(3,\ldots, b)} p_{\sigma(3)} \cdots p_{\sigma(j+1)} + q_{\sigma(3)} \cdots q_{\sigma(j+1)}
\end{align*}
Inequality $(i)$ holds because $p_2, p_3, \ldots, p_b \leq \epsilon$, inequality $(ii)$ because $q_2 \leq q_3 \leq \ldots \leq q_b$, inequality $(iii)$ due to the assumption $\epsilon \leq \frac{1}{j+1}$ and inequality $(iiii)$ since $p_1 \geq 1-\epsilon$.
\end{proof}

\begin{proof}[\textbf{Proof of Lemma \ref{lemma6}}]
Using the definition of $\Psi_j$ in eq. \eqref{eq:defPsi}, inequality (\ref{lemma6ineq}) can be restated by only considering the terms in the summation which differ in the two sides, that is, those corresponding to permutations $\sigma$ such that $1 \in \{ \sigma(1), \ldots, \sigma(j) \}$, $\sigma(j+1)=2$ and $2 \in \{ \sigma(1), \ldots, \sigma(j) \}$, $\sigma(j+1)=1$ and $\{1, 2\} \subseteq \{ \sigma(1), \ldots, \sigma(j) \}$. Therefore we have that
\begin{align*}
&((1-\epsilon+\delta)q_2 + q_1p_2) \sum_{\sigma \in Sym(3,\ldots, b)} p_{\sigma(3)} \cdots p_{\sigma(j+1)} + q_{\sigma(3)} \cdots q_{\sigma(j+1)} + \\
&(j-1)(1-\epsilon+\delta)p_2 \sum_{\sigma \in Sym(3,\ldots, b)} p_{\sigma(3)} \cdots p_{\sigma(j)} q_{\sigma(j+1)}
\\
&<
\left((1-\epsilon)q_2 + q_1(p_2+\delta)\right) \sum_{\sigma \in Sym(3,\ldots, b)} p_{\sigma(3)} \cdots p_{\sigma(j+1)} + q_{\sigma(3)} \cdots q_{\sigma(j+1)} + \\
&(j-1)(1-\epsilon)(p_2+\delta) \sum_{\sigma \in Sym(3,\ldots, b)} p_{\sigma(3)} \cdots p_{\sigma(j)} q_{\sigma(j+1)}.
\end{align*}
That is
\begin{align*}
&\delta (q_1-q_2) \sum_{\sigma \in Sym(3,\ldots, b)} p_{\sigma(3)} \cdots p_{\sigma(j+1)} + q_{\sigma(3)} \cdots q_{\sigma(j+1)} + \\
&(j-1)\delta(1-\epsilon-p_2) \sum_{\sigma \in Sym(3,\ldots, b)} p_{\sigma(3)} \cdots p_{\sigma(j)} q_{\sigma(j+1)}
> 0.
\end{align*}
Which is satisfied because $q_2 < q_1$,  $p_2 < 1-\epsilon$ and $\delta > 0$.
\end{proof}

\end{document}